\documentclass[11pt,reqno,final]{amsart}
\usepackage{color}
\usepackage[colorlinks=true,allcolors=blue,backref=page]{hyperref}
\usepackage{amsmath, amssymb, amsthm}
\usepackage{mathrsfs}
\usepackage{mathtools}
\usepackage[noabbrev,capitalize,nameinlink]{cleveref}
\crefname{equation}{}{}
\usepackage{fullpage}
\usepackage[noadjust]{cite}
\usepackage{graphics}
\usepackage{pifont}
\usepackage{tikz}
\usepackage{bbm}
\usepackage[T1]{fontenc}

\usetikzlibrary{arrows.meta}

\usepackage{environ}
\usepackage{framed}
\usepackage{url}
\usepackage[linesnumbered,ruled,vlined]{algorithm2e}
\usepackage[noend]{algpseudocode}
\usepackage[labelfont=bf]{caption}
\usepackage{cite}
\usepackage{framed}
\usepackage[framemethod=tikz]{mdframed}
\usepackage{appendix}
\usepackage{graphicx}
\usepackage[textsize=tiny]{todonotes}
\usepackage{tcolorbox}
\usepackage{enumerate}
\allowdisplaybreaks[4]
\usepackage{enumerate}
\usepackage{stmaryrd}

\usepackage[shortlabels]{enumitem}
\crefformat{enumi}{#2#1#3}
\crefrangeformat{enumi}{#3#1#4 to~#5#2#6}
\crefmultiformat{enumi}{#2#1#3}
{ and~#2#1#3}{, #2#1#3}{ and~#2#1#3}

\DeclareSymbolFont{symbolsC}{U}{pxsyc}{m}{n}
\SetSymbolFont{symbolsC}{bold}{U}{pxsyc}{bx}{n}
\DeclareFontSubstitution{U}{pxsyc}{m}{n}
\DeclareMathSymbol{\medcircle}{\mathbin}{symbolsC}{7}

\usepackage{thmtools}
\usepackage{thm-restate}

\crefname{algocf}{Algorithm}{Algorithms}

\crefname{equation}{}{} 
\AtBeginEnvironment{appendices}{\crefalias{section}{appendix}} 

\usepackage[color,final]{showkeys} 

\colorlet{refkey}{orange!20}
\colorlet{labelkey}{blue!30}

\crefname{algocf}{Algorithm}{Algorithms}

\numberwithin{equation}{section}
\newtheorem{theorem}{Theorem}[section]
\newtheorem{proposition}[theorem]{Proposition}
\newtheorem{lemma}[theorem]{Lemma}

\crefname{claim}{Claim}{Claims}

\newtheorem{corollary}[theorem]{Corollary}

\newtheorem*{question*}{Question}

\theoremstyle{definition}
\newtheorem{definition}[theorem]{Definition}

\newtheorem*{definition*}{Definition}

\theoremstyle{remark}
\newtheorem{remark}[theorem]{Remark}
\newtheorem*{remark*}{Remark}

\newcommand{\mb}{\mathbb}
\newcommand{\mbf}{\mathbf}
\newcommand{\mbm}{\mathbbm}
\newcommand{\mc}{\mathcal}

\newcommand{\mr}{\mathrm}

\newcommand{\on}{\operatorname}

\newcommand{\eps}{\varepsilon}
\renewcommand{\tilde}{\widetilde}

\allowdisplaybreaks

\title{Subgraph distributions in dense random regular graphs}

\author[A1]{Ashwin Sah}
\author[A2]{Mehtaab Sawhney}
\address{Department of Mathematics, Massachusetts Institute of Technology, Cambridge, MA 02139, USA}
\email{\{asah,msawhney\}@mit.edu}

\thanks{Sah and Sawhney were supported by NSF Graduate Research Fellowship Program DGE-2141064. Sah was supported by the PD Soros Fellowship.}

\begin{document}

\maketitle
\begin{abstract}
Given connected graph $H$ which is not a star, we show that the number of copies of $H$ in a dense uniformly random regular graph is asymptotically Gaussian, which was not known even for $H$ being a triangle. This addresses a question of McKay from the 2010 International Congress of Mathematicians. In fact, we prove that the behavior of the variance of the number of copies of $H$ depends in a delicate manner on the occurrence and number of cycles of $3,4,5$ edges as well as paths of $3$ edges in $H$. More generally, we provide control of the asymptotic distribution of certain statistics of bounded degree which are invariant under vertex permutations, including moments of the spectrum of a random regular graph.

Our techniques are based on combining complex-analytic methods due to McKay and Wormald used to enumerate regular graphs with the notion of graph factors developed by Janson in the context of studying subgraph counts in $\mathbb{G}(n,p)$.
\end{abstract}

\section{Introduction}\label{sec:introduction}
The study of the asymptotic distribution of small subgraph counts in the Erd\H{o}s-R\'enyi random graphs $\mb{G}(n,p)$ and $\mb{G}(n,m)$ has been a topic of central interest in random graph theory. In particular, following a long series of papers, Ruci\'{n}ski \cite{Ru88} established the optimal conditions under which $X_H$, the number of unlabelled copies of $H$ in $\mb{G}(n,p)$, satisfies a central limit theorem. Furthermore, in general the distribution of small subgraphs in $\mb{G}(n,p)$ is known to a substantial degree of precision. We in particular refer the reader to the book \cite{JLR00} and references therein for a more complete account.

With regards to asymptotic distributions, the state of affairs for random $d$-regular graphs is substantially less satisfactory. Let $\mb{G}(n,d)$ denote a uniformly random $d$-regular graph. Note that unlike $\mb{G}(n,p)$ or $\mb{G}(n,m)$, the edges in $\mb{G}(n,d)$ exhibit strong and non-obvious correlations and therefore even the question of determining the number of $d$-regular graphs has a rich history drawing on techniques ranging from switchings developed by McKay \cite{McK85} (and refined by McKay and Wormald \cite{MW91}), a complex-analytic technique of McKay and Wormald \cite{MW90}, and recent breakthroughs using fixed-point iteration due to Liebenau and Wormald \cite{LW17}. We refer the reader to the excellent survey of Wormald \cite{Wor18} where the extensive history of this problem and various related enumeration problems are discussed.

McKay \cite{McK10} in his 2010 ICM survey on graphs with a fixed degree sequence asked for an understanding of the asymptotic distribution of subgraph counts in dense random regular graphs, noting that ``there is almost nothing known about the distribution of subgraph counts'' for these models; the state of affairs has remained unchanged since. In particular, the only result which applies in this regime is work of McKay \cite{McK11} which computes the expectation of the number of subgraphs of a fixed size in $\mb{G}(n,d)$ (see \cite{IM18} for an extension to more exotic degree sequences). Our main result establishes a central limit theorem for counting copies of connected graphs $H$ in $\mb{G}(n,d)$ for $\min(d,n-d)\ge n/\log n$, and a consequence of our methods demonstrates a joint central limit theorem for so-called ``graph factors'' in the sense of Janson \cite{Jan94}. We additionally apply our techniques to show an analogous result for moments of the spectrum.

Though such a result for dense graphs has until now been out of reach, there is a rich literature regarding sparser graphs. A variety of results have been proven based on applications of the moment method and taking sufficiently fast growing moments. When $d$ is constant, the cycle count distribution was shown to asymptotically converge to a Poisson distribution independently by Bollob\'{a}s \cite{Bol80} and Wormald \cite{Wor81}. This result was extended to strictly balanced graphs near the threshold for existence by Kim, Sudakov, and Vu \cite{KSV07}, establishing a Poisson limit theorem in general. For results regarding asymptotic normality, McKay, Wormald, and Wysocka \cite{MWW04} proved asymptotic normality of cycle counts for $d$ tending to infinity sufficiently slowly, in particular proving normality of triangle counts when $d = o(n^{1/5})$. A later result of Gao and Wormald \cite{GW08} improved this result for a variety of subgraph structures $H$ by counting isolated copies, including extending the regime for triangle counts to $d = o(n^{2/7})$. This was further improved by Gao \cite{Gao20} who improved the range of normality for triangle counts to $d = O(n^{1/2})$. Finally, we note that the study of the asymptotic distribution of the number of spanning structures in random regular graphs has also been of interest (see \cite{Gao20} for further discussion).

Before stating our results let us formally define a random graph with a specified degree sequence.

\begin{definition}\label{def:random-reg}
Given nonnegative sequence $\mbf{d}=(d_1,\ldots,d_n)$, let $\mb{G}(\mbf{d})$ be a uniformly random simple graph $G$ with degree sequence $\mbf{d}$. When $2|dn$ let $\mb{G}(n,d)$ be a uniformly random simple graph on $n$ vertices which is $d$-regular, and let $G(n,d)$ be the set of possible outcomes. Given $G\sim\mb{G}(n,d)$ we define its density $p=p(G)=e(G)/\binom{v(G)}{2}=d/(n-1)$.
\end{definition}

Our results provide a complete understanding of the small subgraph distribution for dense random regular graphs. We first state a corollary of our main result regarding the distribution of subgraph statistics in $\mb{G}(n,d)$. Note first that the number of stars with $s\ge 2$ leaves in a $d$-regular graph on $n$ vertices is trivially always $n\binom{d}{s}$, so we exclude this case from consideration. Additionally, given graphs $H$ and $F$ let $N(H,F)$ be the number of unlabelled copies of $F$ in $H$ (or more precisely the number of distinct, not necessarily induced, subgraphs of $H$ which are isomorphic to $F$). In particular for this definition we have that $N(C_5, P_5) = 5$, where we write $C_k$ and $P_k$ for a cycle and path respectively on $k$ vertices.

\begin{theorem}\label{cor:deduction-1}
Fix a nonempty connected graph $H$ which is not a star and let $X_H$ denote the number of unlabelled copies of $H$ in $G\sim\mb{G}(n,d)$. If $n/\log n\le\min(d,n-d)$, $2|dn$, and $G\sim\mb{G}(n,d)$ we have:
\begin{itemize}
    \item If $H$ contains a $C_3$ then 
    \[\bigg(\frac{X_H-\mb{E}X_H}{\sqrt{\mr{Var}[X_H]}}\bigg)\xrightarrow[]{d.}\mc{N}(0,1)\]
    with $\mr{Var}[X_H] = 6N(H,C_3)^2p^{2e(H)-3}(1-p)^3\frac{n^{2v(H)-3}}{\mr{aut}(H)^2} + O(n^{2v(H)-3-1/6})$.
    \item If $H$ contains a $C_4$ and no $C_3$ then 
    \[\bigg(\frac{X_H-\mb{E}X_H}{\sqrt{\mr{Var}[X_H]}}\bigg)\xrightarrow[]{d.}\mc{N}(0,1)\]
    with $\mr{Var}[X_H] = 8N(H,C_4)^2p^{2e(H)-4}(1-p)^4\frac{n^{2v(H)-4}}{\mr{aut}(H)^2}+ O(n^{2v(H)-4-1/6})$.
    \item If $H$ does not contain a $C_3$ or $C_4$ then it contains a $P_4$ and
    \[\bigg(\frac{X_H-\mb{E}[X_H]}{\sqrt{\mr{Var}[X_H]}}\bigg)\xrightarrow[]{d.}\mc{N}(0,1)\]
    with $\mr{Var}[X_H] = (10p^{2e(H)-5}(1-p)^5N(H,C_5)^2 + 6p^{2e(H)-3}(1-p)^3N(H,P_4)^2)\frac{n^{2v(H)-5}}{\mr{aut}(H)^2} +  O(n^{2v(H)-5-1/6})$.
\end{itemize}
\end{theorem}
\begin{remark}
We note that $\mb{E}X_H = (1+o(1))n^{v(H)}p^{e(H)}/\on{aut(H)}$ is known due to \cite{McK11} and in fact our method can be used to compute the expectation to accuracy $o(\sqrt{\mr{Var}[X_H]})$, but the resulting expressions are rather involved. Additionally, the above result implies that for $H$ being a triangle we have that the variance of $X_H$ is on the order of $p^3n^3$ for $1/\log n\le p\le 1/2$, whereas in $\mb{G}(n,p)$ the variance is of order $\max(p^3n^3,p^5n^4)$. Therefore for the range of $p$ under consideration $\mr{Var}[X_H]$ is substantially lower than in $\mb{G}(n,p)$; this is unlike the results of McKay, Wormald, and Wysocka \cite{MWW04} and Gao and Wormald \cite{GW08} when $p$ is sufficiently sparse. We also note that for $H$ not containing a triangle, $\mr{Var}[X_H]$ is in fact asymptotically smaller than the corresponding variance in $\mb{G}(n,m)$. Finally we note that the subgraph counts are not in general asymptotically independent.
\end{remark}

In general our results are sufficiently powerful to deduce the asymptotic distribution of statistics of fixed degree which is invariant under vertex permutation. To state our main result we will need the notion of graph factors as defined by Janson \cite{Jan94}. Let $x_e$ be the indicator random variable for whether an edge $e$ is in included in random graph $G\sim\mb{G}(n,d)$ and let $\chi_e^{} = (x_e-p)/\sqrt{p(1-p)}$. Note that by symmetry, marginally each $x_e$ is distributed as $\mr{Ber}(p)$ and thus $\chi_e^{}$ has mean $0$ and variance $1$. However, as $G$ is a random regular graph there are substantial correlations between different edges $\chi_e^{}$.

\begin{definition}\label{def:graph-factor}
Fix a graph $H$ with no isolated vertices and an integer $n\ge|v(H)|$. Then define
\[\gamma_H(\mbf{x}) = \sum_{\substack{E\subseteq K_n\\E\simeq H}}\prod_{e\in E}\chi_e^{}.\]
Here $\simeq$ denotes graph isomorphism, specifically between $H$ and the graph spanned by the edges $E$. We will frequently adopt the shorthand that $\chi_S^{} = \prod_{e\in S}\chi_e^{}$. We call $\gamma_H(\mbf{x})$ the \emph{graph factor} corresponding to the graph $H$. When $H$ is connected and its minimum degree is at least $2$, define the \emph{normalized graph factor} to be
\[\tilde{\gamma}_H(G) =  (\gamma_H(\mbf{x})-E_{H})/\sigma_{H}\]
where $\sigma_H =\big(\frac{n^{v(H)}}{\mr{aut}(H)}\big)^{1/2}$ and $E_H = 0$ if $H$ is not an even cycle and $E_H = \frac{2n^{v(H)/2}}{\mr{aut}(H)}=\frac{n^{v(H)/2}}{v(H)}$ if it is.
\end{definition}
\begin{remark}
In the original notion \cite{Jan94}, all connected graphs $H$ are needed to express symmetric functions of graphs. However, as we will see in \cref{sec:theory}, $d$-regularity means that $\gamma_H$ for $H$ with a degree $1$ vertex can be expressed as a linear combination of the smaller $\gamma_{H'}$ (in terms of $d$). Additionally, the $E_H$ (approximate expectation) term for even cycles makes an appearance due to regularity. This is another departure from $\mb{G}(n,p)$ behavior, since in the independent setting the expectation of every $\gamma_H(\mbf{x})$ is $0$.
\end{remark}

We now are in position to state our main result.

\begin{theorem}\label{thm:main}
Fix a collection of nonisomorphic connected graphs $\mc{H} = \{H_i\colon 1\le i\le k\}$ each of minimum degree at least $2$. Let $n/\log n\le\min(d,n-d)$, $2|dn$, and $G\sim\mb{G}(n,d)$. Then as $n\to\infty$ (uniformly in $d$), we have
\[\big(\tilde{\gamma}_{H_i}(G)\big)_{1\le i\le k}\xrightarrow[]{d.}\mc{N}(0,1)^{\otimes k}.\]
Furthermore, $\mr{Var}[\gamma_H(G)] = (1+O(n^{-1/6}))\sigma_H^2$ and $\mb{E}\gamma_H(G) = E_H + O(n^{-1/6}\sigma_H)$.
\end{theorem}
\begin{remark}
The convergence in distribution can be made quantitative in terms of Kolmogorov distance or Wasserstein distance by quantifying the convergence in moments (see e.g.~\cite[Theorem~4]{RW19}); however the associated rates are quantitatively quite poor. 
\end{remark}

Although \cref{thm:main} is stated in terms of graph factors, a straightforward computation allows one to deduce the asymptotic distribution of any symmetric statistic of the edges of bounded degree; one can think of these as the ``building blocks'' for all such statistics in $\mb{G}(n,d)$.

Our results also imply that the traces of fixed powers of $A_G$ (the adjacency matrix of $G$), or equivalently the moments of the spectrum, satisfy a joint central limit theorem. Using techniques of Sinai and Soshnikov \cite{SS98} along with a suitable modifications one could likely extend the result to prove normal fluctuations for sufficiently nice test functions (e.g.~analytic functions with suitably large radius of convergence). However, given that substantially stronger results are likely plausible using Green's function estimates established by He \cite{He22} and results connecting such estimates with functional central limit theorems (see \cite{LS20} and reference therein), we omit such an extension. Additionally, we remark that direct spectral techniques are insufficient to recover \cref{thm:main} since graph factors which do not correspond to cycles are not purely determined by the spectrum.

\begin{corollary}\label{cor:deduction-2}
Given $k\ge 3$, there exists a positive definite matrix $\Sigma_k\in\mb{R}^{(k-2)\times(k-2)}$ such that the following holds. For $G\sim\mb{G}(n,d)$ with $n/\log n\le\min(d,n-d)$ and $2|dn$, $E_\ell = \mb{E}(\on{tr}(A_G^\ell))$, and $\sigma_\ell^2=\mr{Var}[\on{tr}(A_G^\ell)]$, we have
\[\big(\sigma_{\ell}^{-1}(\on{tr}(A_G^\ell)-E_{\ell})\big)_{3\le\ell\le k}\xrightarrow[]{d.}\mc{N}(0,\Sigma_k).\]
\end{corollary}
\begin{remark}
Note that $\on{tr}(A_G) = 0$ and $\on{tr}(A_G^2) = dn$ deterministically.
\end{remark}

\subsection{Proof techniques}\label{sub:proof-techniques}
Our proof uses techniques from the enumeration of dense graphs with a specific degree sequence given by McKay and Wormald \cite{MW90} combined with the notion of graph factors introduced by Janson \cite{Jan94}. The crucial technical point is that previous work regarding asymptotic normality relied on computing the raw moments of $X_H$ and therefore requires taking a number of moments which grows with $n$. In our approach, one instead notices that any symmetric statistic on $d$-regular graphs can be expressed in terms of simple building blocks, and we can directly prove a joint central limit theorem for this collection. In order to prove the necessary limit theorem, we only require an arbitrarily slowly growing moment of these graph factors and they are particularly well-behaved when using the complex-analytic techniques developed by McKay and Wormald \cite{MW90}. In particular, the necessary moments of graph factors can be given a natural complex-analytic expression using the multidimensional Cauchy integral formula. Then desired estimates can be computed directly. In fact, certain comparisons to $|G(n,d)|$ simplify the situation, allowing us to avoid repeating a careful saddle point analysis as in the work of McKay and Wormald \cite{MW90}. The nontrivial expectation contributions $E_H$ when $H$ is an even cycle come into play due to counting certain even-power monomials in a polynomial expansion associated to the edges of $H$.

We further believe combining the general method of considering graph factors along with recent work of Liebenau and Wormald \cite{LW17}, which enumerates graphs of degrees of intermediate sparsity, can likely be used to address asymptotic distribution for regular graphs of all sparsities, a direction we plan to pursue in future work. Finally we note that while \cref{cor:deduction-1} can handle subgraph counts of mildly growing size, the understanding of the asymptotic distributions of spanning subgraphs in dense random regular graphs is also of interest. In particular, do analogues of the asymptotic normality results of Janson \cite{Jan94b} regarding the number of perfect matchings in $\mb{G}(n,m)$ exist for $\mb{G}(n,d)$? 

\subsection{Organization}\label{sub:organization}
In \cref{sec:contour} we prove the main estimates regarding the expectation of $\chi_S^{}$ for a fixed set of edges $S$ via contour integration techniques. In \cref{sec:deduction-of-main}, we deduce \cref{thm:main} via the method of moments and a graph-theoretic argument which guarantees that the estimates in \cref{sec:contour} are of sufficient accuracy. In \cref{sec:theory} we develop the theory of graph factors in $d$-regular graphs and prove that any symmetric graph statistic of fixed degree, when evaluated on $d$-regular graphs, can be expressed as a polynomial of graph factors that are of the type described in \cref{def:graph-factor}. Finally in \cref{sec:final} we deduce \cref{cor:deduction-1,cor:deduction-2} as straightforward consequences of our main results and the proofs in \cref{sec:theory}.

\subsection{Notation}\label{sub:notation}
We use standard asymptotic notation throughout, as follows. For functions $f=f(n)$ and $g=g(n)$, we write $f=O(g)$ or $f \lesssim g$ to mean that there is a constant $C$ such that $|f(n)|\le C|g(n)|$ for sufficiently large $n$. Similarly, we write $f=\Omega(g)$ or $f \gtrsim g$ to mean that there is a constant $c>0$ such that $f(n)\ge c|g(n)|$ for sufficiently large $n$. Finally, we write $f\asymp g$ or $f=\Theta(g)$ to mean that $f\lesssim g$ and $g\lesssim f$, and we write $f=o(g)$ or $g=\omega(f)$ to mean that $f(n)/g(n)\to0$ as $n\to\infty$. We write $O_H(1)$ for some unspecified constant that can be chosen as some bounded value depending only on $H$. Additionally we set $k!! = 2^{k/2}\cdot (k/2)!$ for even integers $k\ge 0$. Finally we let $[n] = \{1,\ldots, n\}$ and $\binom{[n]}{2} = \{(i,j): 1\le i<j\le n\}$.

\subsection*{Acknowledgments}
We thank Vishesh Jain for several useful discussions which played a key role in the development of the project. We also thank Brendan McKay and Nick Wormald for helpful comments. Finally, we are very grateful to the anonymous referee for a number of useful comments, including finding errors in the initial version of the manuscript with respect to the statements of the main results, and for providing us with numerical data which helped to corroborate various claims.

\section{Cancellation Estimates Based on Contour Integrals}\label{sec:contour}
\subsection{Preliminary estimates}\label{sub:preliminary-estimates}
We first recall a number of estimates from the work of McKay and Wormald \cite{MW90}.
\begin{lemma}[{\cite[Lemma~1]{MW90}}]\label{lem:exp-estimate}
Let $0\le \lambda\le 1$ and $|x|\le \pi$. Then we have that 
\[|1+\lambda(e^{ix}-1)| = (1-2\lambda(1-\lambda)(1-\cos x))^{1/2}\le \exp\bigg(-\frac{1}{2}\lambda(1-\lambda)x^2+\frac{1}{24}\lambda(1-\lambda)x^4\bigg).\]
\end{lemma}
\begin{lemma}[{\cite[(3.3)]{MW90}}]\label{lem:quadratic-inequality}
We have for $x_j\in \mb{R}$ that
\[\sum_{1\le j<k\le\ell}(x_j+x_k)^2\ge(\ell-2)\sum_{1\le j\le\ell}x_j^2,\qquad\sum_{1\le j<k\le\ell}(x_j+x_k)^4\le 8(\ell-1)\sum_{1\le j\le\ell}x_j^4.\]
\end{lemma}

We also require the following elementary estimate (a variant of which appears in \cite[p.~8]{MW90}); we provide a proof for the sake of completeness. 
\begin{lemma}\label{lem:gaussian-inequality}
We have for $m\ge m_{\ref{lem:gaussian-inequality}}$ that 
\[\int_{-\pi/16}^{\pi/16}\exp(-mx^2 + mx^4)dx = (1\pm 2m^{-1})\sqrt{\pi/m}.\]
\end{lemma}
\begin{proof}
Notice that for $m$ larger than an absolute constant, 
\begin{align*}
\int_{-\pi/16}^{\pi/16}\exp(-mx^2 + mx^4)dx &= \int_{-m^{-2/5}}^{m^{-2/5}}\exp(-mx^2 + mx^4)dx \pm \pi/8\cdot\exp(-m^{1/5})\\
& = \int_{-m^{-2/5}}^{m^{-2/5}}(1\pm 2mx^4)\exp(-mx^2)dx \pm \pi/8\cdot\exp(-m^{1/5})\\
&= (1\pm 2m^{-1})\sqrt{\pi/m}. \qedhere
\end{align*}
\end{proof}

We will also use an elementary estimate bounding large moments in the following twisted Gaussian integral.
\begin{lemma}\label{lem:gaussian-inequality-2}
We have 
\[\int_{-\pi/16}^{\pi/16}|x|^{k}\exp(-mx^2 + mx^4)dx\le\sqrt{2\pi}k^{k/2}m^{-(k+1)/2}.\]
\end{lemma}
\begin{proof}
Note
\begin{align*}
\int_{-\pi/16}^{\pi/16}|x|^{k}\exp(-mx^2 + mx^4)~dx &\le \int_{-\infty}^{\infty}|x|^{k}\exp(-mx^2/2)~dx= m^{-(k+1)/2}\int_{-\infty}^{\infty}|x|^{k}\exp(-x^2/2)~dx  \\
&= \sqrt{2\pi}m^{-(k+1)/2}\mb{E}_{Z\sim\mc{N}(0,1)}|Z|^k\le k^{k/2}\sqrt{2\pi}m^{-(k+1)/2}.\qedhere
\end{align*}
\end{proof}

We will need another polynomial inequality in the real numbers.
\begin{lemma}\label{lem:symmetric-sum-ineq}
For $x_1,\ldots,x_\ell\in\mb{R}$ we have
\[k!\sum_{1\le j_1<\cdots<j_k\le\ell}x_{j_1}^2\cdots x_{j_k}^2\le\Big(\sum_{1\le j\le\ell}x_j^2\Big)^k\le k!\sum_{1\le j_1<\cdots<j_k\le \ell }x_{j_1}^2\cdots x_{j_k}^2+\binom{k}{2}\Big(\max_{j\in[\ell]}x_j^2\Big)\Big(\sum_{1\le j\le\ell}x_j^2\Big)^{k-1}.\]
\end{lemma}
\begin{proof}
The first inequality is trivial. For the second, consider expanding $(\sum_{j=1}^{\ell}x_j^2)^k$ and removing the terms which have no repeated index. For the remaining terms, remove the first term in the sequence that later repeats and bound it by $\max_{j\in[\ell]}x_j^2$. It is easy to check that the resulting map on index sequences has fibers of size at most $\binom{k}{2}$.
\end{proof}

Finally, we require the main result of \cite[Theorem~1]{MW90} which provides a sharp estimate for $|G(n,d)|$.
\begin{theorem}\label{thm:count}
There exists $\eps = \eps_{\ref{thm:count}}>0$ such that for $n/\log n\le\min(d,n-d)$, $2|dn$, $\lambda = d/(n-1)$, and $r = \sqrt{\lambda/(1-\lambda)}$, we have
\begin{align*}
|G(n,d)| &= 2^{1/2}(2\pi\lambda^{d+1}(1-\lambda)^{n-d}n)^{-n/2}\exp\bigg(\frac{-1+10\lambda-10\lambda^2}{12\lambda(1-\lambda)}+O(n^{-\eps})\bigg)\\
&=\frac{(1+r^2)^{\binom{n}{2}}}{(2\pi r^{d})^{n}}\bigg(\frac{2\pi}{\lambda(1-\lambda)n}\bigg)^{n/2}\bigg(2^{1/2}\exp\bigg(\frac{-1+10\lambda -10\lambda^2}{12\lambda(1-\lambda)}+O(n^{-\eps})\bigg)\bigg).
\end{align*}
\end{theorem}

\subsection{Graph factor estimates}\label{sub:graph-factor-estimates}
The crucial estimates for the remainder of the proof will be the following inequalities controlling the behavior of the constituent expectations in a graph factor.

\begin{proposition}\label{prop:graph-factor-computation}
There is $C=C_{\ref{prop:graph-factor-computation}}>0$ so that for a set of distinct edges $S\subseteq K_n$ the following holds. Let $p = d/(n-1)$, $n/\log n\le\min(d,n-d)$, and $2|dn$. Recall the notation $\chi_S^{}$ from \cref{def:graph-factor}. 
\begin{itemize}
    \item For any $S$ such that $|S|\le\sqrt{\log n}$ we have 
    \[|\mb{E}_{G\sim\mb{G}(n,d)}\chi_S^{}|\le Cn^{-|S|/2+1/4}.\]
    \item For $S$ such that $|S|\le\sqrt{\log n}$, and there is a connected component which is an odd cycle, we have 
    \[|\mb{E}_{G\sim\mb{G}(n,d)}\chi_S^{}|\le Cn^{-1/4}n^{-|S|/2}.\]
    \item For $S$ such that $|S|\le\sqrt{\log n}$, the set of edges form a set of vertex disjoint even cycles, and there are $\ell$ disjoint cycles, we have 
    \[|\mb{E}_{G\sim\mb{G}(n,d)}\chi_S^{}-2^\ell n^{-|S|/2}|\le Cn^{-1/5}n^{-|S|/2}.\]
\end{itemize}
\end{proposition}
As mentioned the initial reduction in the proof closely mimics that of the proof of \cite[Theorem~1]{MW90}.
\begin{proof}
Note that by complementing $\mb{G}(n,d)$ and replacing $p$ by $1-p$, we have that 
\[\mb{E}_{G\sim\mb{G}(n,d)}\chi_S^{} = (-1)^{|S|}\mb{E}_{G\sim \mb{G}(n,n-d-1)}\chi_S^{}.\]
Therefore it suffices to treat the case where $p\le 1/2$.

By Cauchy's integral formula and taking the contours for $z_j$ to be circles of radius $r = \sqrt{p/(1-p)}$ around the origin we have 
\begin{align}
&\mb{E}_{G\sim\mb{G}(n,d)}\chi_S^{} = \frac{(2\pi i)^{-n}}{|G(n,d)|}\oint \cdots \oint \frac{\prod_{(j,k)\notin S}(1+z_jz_k)\prod_{(j,k)\in S}(-p+(1-p)z_jz_k)/\sqrt{p(1-p)}}{\prod_{j\in[n]}z_j^{d+1}}dz\notag\\
&= \frac{(1+r^2)^{\binom{n}{2}}}{(2\pi r^{d})^{n}|G(n,d)|}\int_{-\pi}^{\pi} \cdots \int_{-\pi}^{\pi}\frac{\prod_{(j,k)\notin S}(1+p(e^{i(\theta_j+\theta_k)}-1))\prod_{(j,k)\in S}(p(1-p))^{1/2}(e^{i(\theta_j+\theta_k)}-1)}{\exp(id\sum_{j\in [n]}\theta_j)}d\theta\label{eq:cauchy-integral}
\end{align}
where $dz=\prod_{j\in[n]}dz_j$ and $d\theta=\prod_{j\in[n]}d\theta_j$, and the product is over unordered pairs $(j,k)$ which can be thought of as edges of the complete graph $K_n$.

\textbf{Step 1: Localizing $\theta$.}
As in \cite[Theorem~1]{MW90}, which corresponds to the case $S = \emptyset$, the first maneuver is to localize near the origin, and the techniques are similar. Let $t = \pi/8$ and fix $\eps$ to be a small numerical constant to be chosen later ($\eps = 10^{-10}$ will suffice). We divide indices based on where they lie on the circle: $S_1 = \{j\colon\theta_j\in [-t,t]\}$, $S_2 = \{j\colon\theta_j\in[t,\pi - t]\}$, $S_3 = \{j\colon\theta_j\in [\pi-t,\pi] \cup [-\pi,-\pi + t]\}$, and $S_4 = \{j\colon\theta_j\in [-\pi+t,- t]\}$. Let $\mbf{R}$ denote the set of $\theta$ such at least one of $|S_1||S_3|\ge n^{1+\eps}$, $|S_2|^2\ge n^{1+\eps}$, or $|S_4|^2\ge n^{1+\eps}$ holds. We have by \cref{lem:exp-estimate} that
\begin{align}
&\bigg|\int_{\mbf{R}}\frac{\prod_{(j,k)\notin S}(1+p(e^{i(\theta_j+\theta_k)}-1))\prod_{(j,k)\in S}(p(1-p))^{1/2}(e^{i(\theta_j+\theta_k)}-1)}{\exp(id\sum_{j\in [n]}\theta_j)}d\theta\bigg|\notag\\
&\le\int_{\mbf{R}}(1-2p(1-p)(1-\cos(2t)))^{n^{1+\eps}/3-|S|}d\theta\le\exp(-\Omega(n^{1+\eps/2}))\label{eq:remainder-1}
\end{align}
and therefore it will suffice to consider $\theta\notin\mbf{R}$. Thus $|S_2|,|S_4|\le n^{1/2+\eps/2}$. Furthermore note that $\theta\notin\mbf{R}$ implies that $|S_1|\le n^{\eps}$ or $|S_3|\le n^\eps$. As the integrand is invariant under $\theta\to\theta + \pi$ (since $2|dn$) it suffices to consider when $|S_3|\le n^\eps$ and multiply the resulting integral by a factor of $2$.

Let $\mbf{R}'$ denote the set of $\theta$ such that $\theta\notin\mbf{R}$, $|S_3|\le n^{\eps}$, and there is $\theta_j\notin [-n^{-1/2+\eps},n^{-1/2+\eps}]$. We have 
\begin{align}
&\bigg|\int_{\theta\in\mbf{R}'}\frac{\prod_{(j,k)\notin S}(1+p(e^{i(\theta_j+\theta_k)}-1))\prod_{(j,k)\in S}(p(1-p))^{1/2}(e^{i(\theta_j+\theta_k)}-1)}{\exp(id\sum_{j\in [n]}\theta_j)}d\theta\bigg|\notag\\
&\le\int_{\theta\in\mbf{R}'}\prod_{(j,k)\notin S}|1+p(e^{i(\theta_j+\theta_k)}-1)|d\theta\le e^{O(|S|)}\bigg(\frac{2\pi}{\lambda(1-\lambda)n}\bigg)^{n/2}\exp(-\Omega(n^\eps)),\label{eq:remainder-2}
\end{align}
where we used a slight modification of \cite[(3.4),~(3.5)]{MW90} in the second inequality (namely, the analogy to the intermediate upper bound given in \cite{MW90} is multiplicatively stable with respect to removal of the terms corresponding to $(j,k)\in S$).

Finally, let $\mbf{U}$ denote the set of $\theta$ such that $|\theta_j|\le n^{-1/2+\eps}$ for all $j$. Combining \cref{eq:cauchy-integral,eq:remainder-1,eq:remainder-2}, the above symmetry observation, and \cref{thm:count} yields
\begin{align}
&\mb{E}_{G\sim \mb{G}(n,d)}\chi_S^{}\pm\exp(-\Omega(n^\eps))\notag\\
&=\frac{2(1+r^2)^{\binom{n}{2}}}{(2\pi r^{d})^{n}|G(n,d)|}\int_{\mbf{U}}\frac{\prod_{(j,k)\notin S}(1+p(e^{i(\theta_j+\theta_k)}-1))\prod_{(j,k)\in S}(p(1-p))^{1/2}(e^{i(\theta_j+\theta_k)}-1)}{\exp(id\sum_{j\in [n]}\theta_j)}d\theta.\label{eq:U-formula}
\end{align}

\textbf{Step 2: Reducing the $S$ contribution to a polynomial.}
We next apply a Taylor series transformation in order to reduce to a more symmetric integral where the terms depending on $S$ are polynomial factors within the integrand. First notice that if $\theta_j,\theta_k$ are sufficiently small then
\[|(e^{i(\theta_j+\theta_k)}-1)(1+p(e^{i(\theta_j+\theta_k)}-1))^{-1}-i(\theta_j+\theta_k)|\le|\theta_j+\theta_k|^2.\]
Therefore for $\theta\in\mbf{U}$ we have
\[\bigg|\prod_{(j,k)\in S}(e^{i(\theta_j+\theta_k)}-1)(1+p(e^{i(\theta_j+\theta_k)}-1))^{-1} - \prod_{(j,k)\in S}i(\theta_j+\theta_k)\bigg|\le 2^{|S|} (2n^{-1/2+\eps})\prod_{(j,k)\in S}|\theta_j+\theta_k|.\]

Define
\[P_1(\theta) = \prod_{(j,k)\notin S}(1+p(e^{i(\theta_j+\theta_k)}-1))\prod_{(j,k)\in S}(e^{i(\theta_j+\theta_k)}-1)\]
and
\[P_2(\theta) = \prod_{(j,k)\in \binom{[n]}{2}}(1+p(e^{i(\theta_j+\theta_k)}-1))\prod_{(j,k)\in S}(i(\theta_j+\theta_k)).\]
Note the above analysis implies
\begin{align*}
&\bigg|\int_{\mbf{U}}\frac{P_1(\theta)-P_2(\theta)}{\exp(id\sum_{j\in [n]}\theta_j)}d\theta\bigg|\le 2^{|S|}\int_{\mbf{U}}2n^{-1/2+\eps}\prod_{(j,k)\in S}|\theta_j+\theta_k|\prod_{(j,k)\in \binom{[n]}{2}}|1+p(e^{i(\theta_j+\theta_k)}-1)|d\theta\\
&\qquad\le 8^{|S|}n^{-1/2+\eps}\int_{\mbf{U}}|\theta_1|^{|S|}\prod_{(j,k)\in \binom{[n]}{2}}|1+p(e^{i(\theta_j+\theta_k)}-1)|d\theta\\
&\qquad\le 8^{|S|}n^{-1/2+\eps}\int_{\mbf{U}}|\theta_1|^{|S|}\prod_{(j,k)\in\binom{[n]}{2}}\exp\bigg(-\frac{1}{2}p(1-p)(\theta_j+\theta_k)^2+\frac{1}{24}p(1-p)(\theta_j+\theta_k)^4\bigg)d\theta\\
&\qquad\le 8^{|S|}n^{-1/2+\eps}\int_{\mbf{U}}|\theta_1|^{|S|}\exp\bigg(\sum_{1\le j\le n}-(n-2)\frac{p(1-p)}{2}\theta_j^2+(n-1)\frac{p(1-p)}{3}\theta_j^4\bigg)d\theta\\
&\qquad\le 8^{|S|}n^{-1/2+\eps}\int_{\mbf{U}}|\theta_1|^{|S|}\exp\bigg(\sum_{1\le j\le n}-(n-2)\frac{p(1-p)}{2}\theta_j^2+(n-2)\frac{p(1-p)}{2}\theta_j^4\bigg)d\theta\\
&\qquad\lesssim 16^{|S|}n^{-1/2+\eps}n^{-|S|/2}|S|^{|S|/2}(2\pi/(p(1-p)n))^{n/2}
\end{align*}
where we have applied \cref{lem:exp-estimate,lem:quadratic-inequality,lem:gaussian-inequality,lem:gaussian-inequality-2}. By \cref{eq:U-formula} and \cref{thm:count} it follows that
\begin{equation}\label{eq:poly-formula}
\mb{E}_{G\sim \mb{G}(n,d)}\chi_S^{} = \frac{2(1+r^2)^{\binom{n}{2}}}{(2\pi r^{d})^{n}|G(n,d)|}\int_{\mbf{U}}\frac{P_2(\theta)(p(1-p))^{|S|/2}}{\exp(id\sum_{j\in [n]}\theta_j)}d\theta\pm n^{-(|S|+1)/2 + 2\eps}.
\end{equation}

\textbf{Step 3: Uniform bound on the integral.}
We now prove the first bullet point in \cref{prop:graph-factor-computation}. Note that \cref{lem:exp-estimate,lem:quadratic-inequality} gives
\begin{align*}
&\bigg|\int_{\mbf{U}}\frac{P_2(\theta)}{\exp(id\sum_{j\in [n]}\theta_j)}d\theta\bigg|\le \int_{\mbf{U}}\prod_{(j,k)\in S}|\theta_j+\theta_k|\prod_{(j,k)\in\binom{[n]}{2}}|1+p(e^{i(\theta_j+\theta_k)}-1)|d\theta\\
&\qquad\le 2^{|S|}\int_{\mbf{U}}|\theta_1|^{|S|}\prod_{(j,k)\in\binom{[n]}{2}}\exp\bigg(-\frac{1}{2}p(1-p)(\theta_j+\theta_k)^2+\frac{1}{24}p(1-p)(\theta_j+\theta_k)^4\bigg)d\theta\\
&\qquad\le 2^{|S|}\int_{\mbf{U}}|\theta_1|^{|S|}\exp\bigg(\sum_{1\le j\le n}-(n-2)\frac{p(1-p)}{2}\theta_j^2+(n-1)\frac{p(1-p)}{3}\theta_j^4\bigg)d\theta\\
&\qquad\le 2^{|S|}\int_{\mbf{U}}|\theta_1|^{|S|}\exp\bigg(\sum_{1\le j\le n}-(n-2)\frac{p(1-p)}{2}(\theta_j^2-\theta_j^4)\bigg)d\theta\\
&\qquad\lesssim n^{-|S|/2}(4|S|)^{|S|}(2\pi/(p(1-p)n)))^{n/2}
\end{align*}
which immediately gives the desired initial estimate noting that the final term in enumeration count from \cref{thm:count} is bounded by $n^{1/5}$ and since $|S|$ is small.

\textbf{Step 4: Cancellation from odd degree terms.}
We next prove that any polynomial factor in terms of the $\theta$ coefficients which is not an even polynomial exhibits additional cancellation. This will immediately imply the second bullet point as there are at most $2^{|S|}$ terms in $\prod_{(j,k)\in S}(\theta_j+\theta_k)$ and since there is an odd cycle component (implying every term has an index of degree $1$). In particular it suffices to bound 
\[\int_{\mbf{U}}\frac{\prod_{j\in[k]}\theta_j^{\ell_j}\prod_{(j,k)\in\binom{[n]}{2}}(1+p(e^{i(\theta_j+\theta_k)}-1))}{\exp(id\sum_{j\in[n]}\theta_j)}d\theta\]
where $\ell_k = 1$, and $k\le 2|S|$. For this, notice that by symmetry
\begin{align*}
&\bigg|\int_{\mbf{U}}\frac{\prod_{j\in [k]}\theta_j^{\ell_j}\prod_{(j,k)\in\binom{[n]}{2}}(1+p(e^{i(\theta_j+\theta_k)}-1))}{\exp(id\sum_{j\in [n]}\theta_j)}d\theta\bigg| \\
&= \frac{1}{n-k+1}\bigg|\int_{\mbf{U}}\frac{(\sum_{k\le j\le n}\theta_j)\prod_{j\in[k-1]}\theta_j^{\ell_j}\prod_{(j,k)\in\binom{[n]}{2}}(1+p(e^{i(\theta_j+\theta_k)}-1))}{\exp(id\sum_{j\in[n]}\theta_j)}d\theta\bigg|\\
&\le \frac{1}{n-k+1}\int_{\mbf{U}}\Big|\sum_{k\le j\le n}\theta_j\Big|\prod_{j\in [k-1]}|\theta_j|^{\ell_j}\prod_{(j,k)\in\binom{[n]}{2}}\big|1+p(e^{i(\theta_j+\theta_k)}-1)\big|d\theta\\
&\le \frac{2}{n}\int_{\mbf{U}}\Big|\sum_{k\le j\le n}\theta_j\Big|\prod_{j\in [k-1]}|\theta_j|^{\ell_j}\exp\bigg(\sum_{1\le j\le n}-(n-2)\frac{p(1-p)}{2}(\theta_j^2-\theta_j^4)\bigg)d\theta\\
&= \frac{2}{n}\int_{\mbf{U}}\mb{E}_{s\sim\mr{Rad}^{\otimes n}}\Big|\sum_{k\le j\le n}s_j\theta_j\Big|\prod_{j\in [k-1]}|\theta_j|^{\ell_j}\exp\bigg(\sum_{1\le j\le n}-(n-2)\frac{p(1-p)}{2}(\theta_j^2-\theta_j^4)\bigg)d\theta\\
&\le\frac{2}{n}\int_{\mbf{U}}\bigg(\mb{E}_{s\sim\mr{Rad}^{\otimes n}}\Big(\sum_{k\le j\le n}s_j\theta_j\Big)^2\bigg)^{1/2}\prod_{j\in [k-1]}|\theta_j|^{\ell_j}\exp\bigg(\sum_{1\le j\le n}-(n-2)\frac{p(1-p)}{2}(\theta_j^2-\theta_j^4)\bigg)d\theta\\
&\le\frac{2}{n}\int_{\mbf{U}}\bigg(\sum_{k\le j\le n}\theta_j^2\bigg)^{1/2}\prod_{j\in [k-1]}|\theta_j|^{\ell_j}\exp\bigg(\sum_{1\le j\le n}-(n-2)\frac{p(1-p)}{2}(\theta_j^2-\theta_j^4)\bigg)d\theta\\
&\le\frac{2}{n^{1-\eps}}\int_{\mbf{U}}\prod_{j\in[k-1]}|\theta_j|^{\ell_j}\exp\bigg(\sum_{1\le j\le n}-(n-2)\frac{p(1-p)}{2}(\theta_j^2-\theta_j^4)\bigg)d\theta\\
&\lesssim n^{-\sum_{j\in [k]}\ell_j/2-1/2+2\eps}(2\pi/(p(1-p)n)))^{n/2}
\end{align*}
as desired, where in the last line we apply \cref{lem:gaussian-inequality-2} and use that $|S|$ is small.

\textbf{Step 5: Even cycles.}
We now handle the third bullet point, proving that the integral is sufficiently close to the desired quantity. Using the technique in the previous step, and noting that given a set $S$ of $\ell$ disjoint even cycles there are $2^\ell$ terms in the expansion of $\prod_{(j,k)\in S}(\theta_j+\theta_k)$ where every vertex has even degree, we have 
\begin{align}
\bigg|\mb{E}_{G\sim\mb{G}(n,d)}\chi_S^{}-&\frac{2^{\ell + 1}(1+r^2)^{\binom{n}{2}}(p(1-p))^{|S|/2}}{(2\pi r^{d})^{n}|G(n,d)|}\int_{\mbf{U}}\frac{\prod_{j\in [|S|/2]}\theta_j^2\prod_{(j,k)\in\binom{[n]}{2}}(1+p(e^{i(\theta_j+\theta_k)}-1))}{\exp(id\sum_{j\in [n]}\theta_j)}d\theta\bigg|\notag\\
&\lesssim n^{-1/4-|S|/2}.\label{eq:even-cycle-equation}
\end{align}

Notice that $\mb{E}_{G\sim\mb{G}(n,d)}\chi_\emptyset^{} = 1$ by definition and \cref{eq:even-cycle-equation} applies with $S$ empty. Subtracting, it therefore suffices to prove 
\[\bigg|\frac{2^{\ell + 1}(1+r^2)^{\binom{n}{2}}}{(2\pi r^{d})^n|G(n,d)|}\int_{\mbf{U}}\frac{\big(\prod_{j\in [|S|/2]}\theta_j^2-(p(1-p)n)^{-|S|/2}\big)\prod_{(j,k)\in\binom{[n]}{2}}(1+p(e^{i(\theta_j+\theta_k)}-1))}{\exp(id\sum_{j\in [n]}\theta_j)}d\theta\bigg|\lesssim n^{-1/4-|S|/2}.\]
From \cref{lem:symmetric-sum-ineq} we have
\[\frac{(\sum_{j}x_j^2)^{k} - k^2(\max_jx_j^2)(\sum_{j}x_j^2)^{k-1}}{k!}\le \sum_{j_1<\cdots<j_k}x_{j_1}^2\cdots x_{j_k}^2\le \frac{(\sum_{j}x_j^2)^{k}}{k!},\]
and using our initial bounds from earlier it follows immediately that 
\[\bigg|\frac{2^{\ell + 1}(1+r^2)^{\binom{n}{2}}\binom{n}{|S|/2}^{-1}}{(2\pi r^{d})^{n}|G(n,d)|(|S|/2)!}\int_{\mbf{U}}\frac{\big((|S|/2)^2n^{-1+2\eps}(\sum_{j}\theta_j^{2})^{|S|/2-1}\big)\prod_{(j,k)\in\binom{[n]}{2}}(1+p(e^{i(\theta_j+\theta_k)}-1))}{\exp(id\sum_{j\in [n]}\theta_j)}d\theta\bigg|\lesssim n^{-1/4-|S|/2}.\]

Therefore, symmetrizing over all permutations of $[n]$ and trivially bounding some lower order contributions, we see it suffices to bound
\[\frac{2^{\ell+1}(1+r^2)^{\binom{n}{2}}\binom{n}{|S|/2}^{-1}}{(2\pi r^{d})^{n}|G(n,d)|(|S|/2)!}\int_{\mbf{U}}\frac{\big((\sum_{1\le j\le n}\theta_j^2)^{|S|/2}-(p(1-p))^{-|S|/2}\big)\prod_{(j,k)\in\binom{[n]}{2}}(1+p(e^{i(\theta_j+\theta_k)}-1))}{\exp(id\sum_{j\in [n]}\theta_j)}d\theta.\]

Notice that, once again,
\begin{align}
&\int_{\mbf{U}}\frac{\big((\sum_{1\le j\le n}\theta_j^2)^{|S|/2}-(p(1-p))^{-|S|/2}\big)\prod_{(j,k)\in\binom{[n]}{2}}(1+p(e^{i(\theta_j+\theta_k)}-1))}{\exp(id\sum_{j\in [n]}\theta_j)}d\theta\notag\\
&\le\int_{\mbf{U}}\big|\big(\sum_{1\le j\le n}\theta_j^2\big)^{|S|/2}-(p(1-p))^{-|S|/2}\big|\prod_{(j,k)\in\binom{[n]}{2}}\big|1+p(e^{i(\theta_j+\theta_k)}-1)\big|d\theta\notag\\
&\le\int_{\mbf{U}}\big|\big(\sum_{1\le j\le n}\theta_j^2\big)^{|S|/2}-(p(1-p))^{-|S|/2}\big|\exp\bigg(\sum_{1\le j\le n}-(n-2)\frac{p(1-p)}{2}(\theta_j^2-\theta_j^4)\bigg)d\theta.\label{eq:square-damping}
\end{align}
We now proceed via splitting \cref{eq:square-damping} based on the size of $\sum_{1\le j\le n}\theta_j^2$. Defining the region $\mbf{S} = \{\theta\colon|\sum_{1\le j\le n}\theta_j^2-(p(1-p))^{-1}|\ge n^{-1/3}\}$ we have that 
\begin{align*}
&\int_{\mbf{U}}\mbm{1}_{\theta\in\mbf{S}}\big|\big(\sum_{1\le j\le n}\theta_j^2\big)^{|S|/2}-(p(1-p))^{-|S|/2}\big|\exp\bigg(\sum_{1\le j\le n}-(n-2)\frac{p(1-p)}{2}(\theta_j^2-\theta_j^4)\bigg)d\theta\\
&\le\int_{\mbf{U}}\mbm{1}_{\theta\in\mbf{S}}(2n^{2\eps})^{|S|/2}\exp\bigg(\sum_{1\le j\le n}-(n-2)\frac{p(1-p)}{2}(\theta_j^2-\theta_j^4)\bigg)d\theta\\
&\le\int_{\mbf{U}}\mbm{1}_{\theta\in\mbf{S}}(2n^{2\eps})^{|S|/2}\exp\bigg(\sum_{1\le j\le n}\frac{-(n-2n^{2\eps})p(1-p)}{2}\theta_j^2\bigg)d\theta\\
&\le (2n^{2\eps})^{|S|/2} (1/(p(1-p)(n-2n^{2\eps})))^{n/2}\int_{\mb{R}^{n}}\mbm{1}_{|\sum_{j\in [n]}x_j^2-(n-2n^{2\eps})|\ge p(1-p)n^{2/3}/2}\exp\bigg(-\frac{1}{2}\sum_{j\in [n]}x_j^2\bigg)dx\\
&\le\exp(O(n^{2\eps}))\cdot(2\pi/(p(1-p)n))^{n/2}\mb{P}_{Z\sim \mc{N}(0,1)^{\otimes n}}\big[\big|\sum_{1\le j\le n}Z_j^2- n\big|\ge n^{3/5}\big]\\
&\le\exp(n^{-1/10})\cdot(2\pi/(p(1-p)n))^{n/2}
\end{align*}
which is sufficiently small as desired. For the remaining portion of \cref{eq:square-damping} notice that 
\begin{align*}
&\int_{\mbf{U}}\mbm{1}_{\theta\notin\mbf{S}}\big|\big(\sum_{1\le j\le n}\theta_j^2\big)^{|S|/2}-(p(1-p))^{|S|/2}\big|\exp\bigg(\sum_{1\le j\le n}-(n-2)\frac{p(1-p)}{2}(\theta_j^2-\theta_j^4)\bigg)d\theta\\
&\le\int_{\mbf{U}}\mbm{1}_{\theta\notin\mbf{S}}(4p(1-p))^{|S|/2}(n^{-1/3})\big|\exp\bigg(\sum_{1\le j\le n}-(n-2)\frac{p(1-p)}{2}(\theta_j^2-\theta_j^4)\bigg)d\theta\\
&\le\int_{\mbf{U}}(4p(1-p))^{|S|/2}(n^{-1/3})\big|\exp\bigg(\sum_{1\le j\le n}-(n-2)\frac{p(1-p)}{2}(\theta_j^2-\theta_j^4)\bigg)d\theta\\
&\lesssim n^{-2/7}(2\pi/(p(1-p)n))^{n/2}
\end{align*}
where we have used \cref{lem:gaussian-inequality} in the final step. The desired result follows immediately.
\end{proof}

\section{Deduction of \texorpdfstring{\cref{thm:main}}{}}\label{sec:deduction-of-main}
In order to prove \cref{thm:main} we proceed via the method of moments. We will require the following standard result regarding converting control on moments to distributional control. This follows immediately from the standard univariate method of moments via the Cram\'{e}r--Wold device (see e.g.~\cite[Theorem~3.10.6]{Dur19}) which shows that in order to prove convergence of a sequence of random variables $X_n \to \mu\in \mb{R}^d$ in distribution, it suffices to prove convergence of $X_n\cdot \theta\to \mu\cdot \theta$ for all $\theta\in \mb{R}^d$. (A proof of the univariate case of the method of moments is standard; see e.g.~\cite[Section~3.3.5,~Theorem~3.3.25]{Dur19}.)

\begin{lemma}\label{lem:method-of-moments}
Fix a vector $\mu\in \mb{R}^{d}$ and a positive definite matrix in $\Sigma\in \mb{R}^{d\times d}$. Given a sequence of random vectors $X_n\in\mb{R}^{d}$, suppose that for any sequence of nonnegative integers $(\ell_i)_{1\le i\le d}$ that
\[\mb{E}\bigg[\prod_{i=1}^{d}((X_n)_i)^{\ell_i}\bigg] \to \mb{E}_{G\sim \mc{N}(\mu, \Sigma)}\bigg[\prod_{i=1}^{d}(G_i)^{\ell_i}\bigg]\]
as $n\to\infty$. Then it follows that 
\[X_n\xrightarrow[]{d.}\mc{N}(\mu,\Sigma).\]
\end{lemma}

We also require the following graph-theoretic input which will be used when applying the method of moments. For a multigraph $G$, let $E_\mr{sing}(G)$ be the set of edges of multiplicity $1$.

\begin{lemma}\label{lem:graph-theory}
Let $\mc{H} = (H_i)_{1\le i\le k}$ be a sequence of connected graphs each of minimum degree at least $2$ (not necessarily distinct). Consider overlaying the $H_i$ in order to obtain a multigraph $G$ (so overlaying two or more edges would give the corresponding multiplicity in $G$). Let $E_\mr{sing}=E_\mr{sing}(G)$. Then we have
\[v(G)-\frac{1}{2}|E_\mr{sing}|\le\frac{1}{2}\sum_{i=1}^kv(H_i)\]
with equality if and only if each connected component of $G$ is either (i) a cycle of multiplicity $1$ which is isolated or (ii) a multigraph with all multiplicities $2$. Furthermore, in (i) the said connected component arises from a single $H_i$ which is a cycle while in (ii) the connected component arises from two $H_i,H_j$ which are isomorphic and perfectly overlaid.
\end{lemma}
\begin{proof}
Note that $E_\mr{sing}$ contains precisely the edges of $G$ without multiplicity, which therefore arise from a single graph in $\mc{H}$. Additionally, it trivially suffices to prove the claim for each connected component of $G$ individually. Equivalently, we may assume $G$ is connected. We have
\begin{align*}
\frac{1}{2}\sum_{i=1}^kv(H_i)&-v(G)+\frac{|E_\mr{sing}|}{2}=\sum_{v\in V(G)}\bigg(\bigg(\frac{1}{2}\sum_{i=1}^k\mbm{1}_{v\in H_i}\bigg)-1\bigg)+\frac{|E_\mr{sing}|}{2}\\
&\ge -\frac{1}{2}\sum_{v\in V(G)}\mbm{1}[|\{i\in[k]\colon v\in H_i\}| = 1]+\frac{|E_\mr{sing}|}{2}\\
&\ge-\frac{1}{2}\sum_{v\in V(G)}\mbm{1}[|\{i\in[k]\colon v\in H_i\}| = 1]+\sum_{v\in V(G)}\frac{2\mbm{1}[|\{i\in [k]\colon v\in H_i\}| = 1]}{4}\ge 0.
\end{align*}
In the last line, the first inequality is justified as follows: consider distributing a mass of $1/2$ on each edge in $E_\mr{sing}$ into $1/4$ on both its vertices. Note that every vertex that appears in exactly one $H_i$ must be contributed by at least $2$ such edges, since the minimum degree is at least $2$ and such edges clearly must be singletons.

For equality to occur notice that every vertex must have $\mc{H}$-multiplicity $1$ or $2$ (i.e.~appears in $1$ or $2$ of the $H_i$), each singleton edge must occur between two vertices of $\mc{H}$-multiplicity $1$, and no $\mc{H}$-multiplicity $1$ vertex has degree larger than $2$. Notice that as we assumed $G$ is connected, we must have that either all vertices have $\mc{H}$-multiplicity $1$ or all have $\mc{H}$-multiplicity $2$: if there is an edge between the two different types of vertex then it must be a singleton (since one of the endpoints is $\mc{H}$-multiplicity $1$) and hence we have a contradiction to the required property of singleton edges in the equality case. Now, if all vertices are $\mc{H}$-multiplicity $1$, then $G$ must arise from a single graph $H_1$, and the equality case is immediately seen to be a cycle using our assumption that $G=H_1$ is connected and also that every vertex has degree exactly $2$.

We now focus on the complementary case that every vertex has $\mc{H}$-multiplicity $2$ and hence there are no singleton edges. Consider a vertex $v$ of $G$ and suppose without loss of generality that it is in $H_1$ and $H_2$. Every $G$-neighbor $w$ of $v$ has the property that edge $(v,w)$ is not singleton, which implies this edge must be present in both $H_1$ and $H_2$. Thus $w$ is in $H_1$ and $H_2$ (and no other $H_i$ since it has $\mc{H}$-multiplicity $2$). Iterating this argument, and using that $G$ is connected, we see that every vertex of $G$ is in $H_1$ and $H_2$, implying that $k=2$. Since there are no singleton edges, this must be a direct overlay of equal graphs.

Finally, we easily check that (i) and (ii) are easily seen to indeed give equality.
\end{proof}

We now prove \cref{thm:main}. Given the results proven so far this is essentially a routine computation with the method of moments. 

\begin{proof}[Proof of \cref{thm:main}]
Fix a collection of connected graphs $\mc{H} = \{H_i\colon 1\le i\le k\}$ of minimum degree at least $2$. In order to apply the method of moments consider fixed values $\ell_1,\ldots,\ell_k$ and write 
\begin{equation}\label{eq:gamma-moment}
\mb{E}_{G\sim\mb{G}(n,d)}\bigg[\prod_{i=1}^k\gamma_{H_i}(G)^{\ell_i}\bigg] = \bigg(\prod_{i=1}^k\sigma_{H_i}^{-\ell_i}\bigg)\sum_{\substack{1\le i\le k\\1\le j\le\ell_i\\H_{i,j}\simeq H_i}}\mb{E}_{G\sim\mb{G}(n,d)}\bigg[\prod_{i=1}^k\prod_{j=1}^{\ell_i}\chi_{H_{i,j}}^{}\bigg].
\end{equation}
Here the $H_{i,j}$ are embedded into $K_n$, and we are summing over possible simultaneous choices of such unlabeled copies. Recall the definition of $E_H,\sigma_H$ from \cref{def:graph-factor}, and note this means $\prod_{i=1}^k\sigma_{H_i}(G)^{\ell_i} = \Theta(n^{\sum_{i=1}^k\ell_iv(H_i)/2}) = \Theta(n^{\sum_{i,j}v(H_{i,j})/2})$.

We consider the terms based on the isomorphism class of $G = \bigcup_{1\le i\le k}\bigcup_{1\le j\le\ell_i}H_{i,j}$, treating $G$ as a multigraph. Let $E_\mr{sing}=E_\mr{sing}(G)$, the set of singleton edges in the isomorphism class. Notice that the contribution of terms based on $G$ is bounded by $O(n^{v(G)}n^{-|E_\mr{sing}|/2 + 1/3})$ using the first bullet of \cref{prop:graph-factor-computation} (and using that if an edge is repeated multiple times, the corresponding $\chi_e^t$ term can be reduced to a linear combination of $1,\chi_e^{}$ with coefficients depending only on $t$ and $p$). Thus if $v(G)-|E_\mr{sing}|/2<\sum_{i,j}v(H_{i,j})/2$ the terms contribute negligibly (namely, $O(n^{-1/6})$) to the quantity \cref{eq:gamma-moment}, since this implies $v(G)-|E_\mr{sing}|/2\le-1/2+\sum_{i,j}v(H_{i,j})/2$.

But recall that by \cref{lem:graph-theory} we have $v(G)-|E_\mr{sing}|/2\le \sum_{i,j}v(H_{i,j})/2$, and equality occurs only in certain specialized cases where each graph $H_{i,j}$ either (i) is a cycle and its vertices are not used by any other $H_{i',j'}$ or (ii) is perfectly overlaid with another $H_{i,j'}$ (with the same isomorphism type) as equal copies. The earlier analysis shows we may restrict attention to such equality cases, so now we more closely characterize which such terms contribute. Without loss of generality, let us assume that $H_1,\ldots,H_m$ are cycles, if any, while $H_{m+1},\ldots,H_k$ are not cycles. Notice that if any $\ell_t$ for $t\in[m+1,k]$ is odd then it is impossible to pair up and overlay all the $H_{t,i}$ for $i\in[\ell_t]$. This violates the equality condition, so is not possible. Thus if $\ell_t$ for some $t\in[m+1,k]$ is odd, then the total contribution to \cref{eq:gamma-moment} is $O(n^{-1/6})$.

Now consider $H_t$ with $1\le t\le m$. If $H_t$ is an odd cycle and is not overlaid with another, then by (i) above it is isolated within $G$. The second bullet of \cref{prop:graph-factor-computation} again shows the total contribution of terms with such an unpaired $H_t$ to \cref{eq:gamma-moment} is $O(n^{-1/6})$.

Finally, we have a situation where all graphs $H_{i,j}$ except even cycles must be paired among themselves and the even cycles $H_{i,j}$ are either isolated in $G$ or paired with another even cycle $H_{i,j'}$ of the same size and overlaid. Without loss of generality let $H_1,\ldots,H_{m'}$ be the even cycles, if any.

The number of choices for pairing up the graphs other than even cycles is $\prod_{i=m'+1}^k\ell_i!!$. The number of choices for pairing up $s_i\le\ell_i/2$ even cycles for $i\in[m']$ is $\prod_{i=1}^{m'}\binom{\ell_i}{2s_i}(2s_i)!!$. In such a pairing, let $\mc{U}_i\subseteq[\ell_i]$ be the list of unpaired indices for $i\in[m']$. We find that $\prod_{i=1}^k\prod_{j=1}^{\ell_i}\chi_{H_{i,j}}^{}$ is a product of various terms of the form $\chi_e^{}$ for $e\in H_{i,j}$ where $i\in[m']$ and $j\in\mc{U}_i$, as well as terms of the form $\chi_e^2$ in certain connected components of $G$. There are $\sum_{i=1}^{m'}(\ell_i-2s_i)v(H_i)$ vertices of the former type and $\sum_{i=1}^{m'}s_iv(H_i)+\sum_{i=m'+1}^k(\ell_i/2)v(H_i)$ of the latter type. Note that $\chi_e^2=1-(2p-1)\chi_e^{}/\sqrt{p(1-p)}$, and expanding out the repeated terms in such a way yields one term of the form $\prod_{i=1}^{m'}\prod_{j\in\mc{U}_i}\chi_{H_{i,j}}^{}$ and others with additional terms of the form $(2p-1)\chi_e^{}/\sqrt{p(1-p)}$ tacked on. The contribution of such other terms totals at most, by the first bullet of \cref{prop:graph-factor-computation},
\[O\Big(n^{\sum_{i=1}^{m'}(\ell_i-s_i)v(H_i)+\sum_{i=m'+1}^k(\ell_i/2)v(H_i)}\cdot n^{-\sum_{i=1}^{m'}\sum_{j\in\mc{U}_i}e(H_{i,j})/2-1/2+1/3}\Big).\]
The exponent is bounded by $v(G)/2-1/3+\sum_{i=1}^{m'}\sum_{j\in\mc{U}_i}(v(H_{i,j})-e(H_{i,j})/2) = \sum_{i=1}^k\ell_iv(H_i)/2-1/6$ since $H_{i,j}$ for $i\in[m']$ is a cycle, so in \cref{eq:gamma-moment} this amounts to a total contribution of $O(n^{-1/6})$.

Finally, what remains is
\[\mb{E}_{G\sim\mb{G}(n,d)}\bigg[\prod_{i=1}^k\gamma_{H_i}(G)^{\ell_i}\bigg] = \bigg(\prod_{i=1}^k\sigma_{H_i}^{-\ell_i}\prod_{i=m'+1}^k\ell_i!!\bigg)\sum_{s_i\le\ell_i/2}{\sideset{}{^\ast}\sum_{H_{i,j}}}\mb{E}\prod_{i=1}^{m'}\bigg(\binom{\ell_i}{2s_i}(2s_i)!!\prod_{j=2s_i+1}^{\ell_i}\chi_{H_{i,j}}^{}\bigg) + O(n^{-1/6}),\]
where $\sum_{s_i\le\ell_i/2}$ denotes a simultaneous choice of such nonnegative integers $s_i$ for $i\in[m']$ and where ${\sum}^\ast$ denotes a sum over choices of $H_{i,j}$ such that they are all vertex-disjoint other than pairs $H_{i,2j-1}=H_{i,2j}$ for $1\le j\le s_i/2$ when $1\le i\le m'$ as well as for $1\le j\le\ell_i/2$ when $m'+1\le i\le k$. This equation basically means that we can validly pair up the necessary graphs and then replace the $\chi_e^2$ terms by $1$. Furthermore, it is not hard to see based on the considerations so far that we can remove the vertex-disjointness condition between different $H_{i,j}$ without changing the error rate, and thus we can write
\begin{align*}
\mb{E}_{G\sim\mb{G}(n,d)}\bigg[\prod_{i=1}^k\gamma_{H_i}(G)^{\ell_i}\bigg] &= \bigg(\prod_{i=1}^k\sigma_{H_i}^{-\ell_i}\prod_{i=m'+1}^k\ell_i!!\bigg(\binom{n}{v(H_i)}\frac{v(H_i)!}{\mr{aut}(H_i)}\bigg)^{\ell_i/2}\bigg)\times\\
&\sum_{s_i\le\ell_i/2}\bigg(\prod_{i=1}^{m'}\binom{\ell_i}{2s_i}(2s_i)!!\bigg(\binom{n}{v(H_i)}\frac{v(H_i)!}{\mr{aut}(H_i)}\bigg)^{s_i}(\mb{E}\gamma_{H_i})^{\ell_i-2s_i}\bigg) + O(n^{-1/6}),\\
&= \prod_{i=m'+1}^k\ell_i!!\sum_{s_i\le\ell_i/2}\prod_{i=1}^{m'}\binom{\ell_i}{2s_i}(2s_i)!!(\sigma_{H_i}^{-1}\mb{E}\gamma_{H_i}(G))^{\ell_i-2s_i}+O(n^{-1/6}),
\end{align*}
using the formula for $\sigma_H$ in the second step. Finally, the third bullet of \cref{prop:graph-factor-computation} shows that $\mb{E}\gamma_{H_i}(G)=(1+O(n^{-1/5}))2n^{-v(H_i)/2}\cdot\binom{n}{v(H_i)}\frac{v(H_i)!}{\mr{aut}(H_i)}=(1+O(n^{-1/5}))E_H$ for $i\in[m']$. We also know that $E_H=(2/v(H))^{1/2}\sigma_H$ hence we find
\[\mb{E}_{G\sim\mb{G}(n,d)}\bigg[\prod_{i=1}^k\gamma_{H_i}(G)^{\ell_i}\bigg] = \prod_{i=m'+1}^k\ell_i!!\sum_{s_i\le\ell_i/2}\prod_{i=1}^{m'}\binom{\ell_i}{2s_i}(2s_i)!!(E_H/\sigma_H)^{\ell_i-2s_i}+O(n^{-1/6}),\]
which can be seen to match the moments of $\mc{N}(E_H/\sigma_H,1)^{\otimes m'}\otimes\mc{N}(0,1)^{\otimes(k-m')}$. Using \cref{lem:method-of-moments} and shifting appropriately, this implies the desired
\[\big(\tilde{\gamma}_{H_i}(G)\big)_{1\le i\le k}\xrightarrow[]{d.}\mc{N}(0,1)^{\otimes k}.\]
Finally, we briefly note that the moment computations above where $\ell_i\in\{1,2\}$ and all $\ell_j=0$ for $j\neq i$ show that the means and variances are as claimed.
\end{proof}

\section{Computations with Graph Factors}\label{sec:theory}
We now prove that any fixed degree polynomial in the indicator functions $x_e\in\{0,1\}$ which is symmetric under vertex permutation can be rewritten (so that it agrees on the set of $d$-regular graphs) as a function of connected graph factors of the form in \cref{def:graph-factor}. The reduction specifically to connected graph factors appears essentially in the work of Janson \cite[p.~347]{Jan95}.

\begin{lemma}\label{lem:connected-reduction}
Given a disconnected graph $H$ (with no isolated vertices) with connected components $H_1,\ldots,H_k$, $\gamma_{H}(\mbf{x})-\prod_{i=1}^k\gamma_{H_i}(G)$ can be expressed (as a function on graphs) as a sum of $\gamma_{H'}$ with $v(H')<v(H)$ (though $H'$ may be itself disconnected). Furthermore, the coefficients of the sum are bounded by $O(1/(p(1-p))^{O_H(1)})$.
\end{lemma}

This can clearly be inductively applied to show that the connected graph factors generate all graph factors using polynomial expressions. The crucial lemma for our work is that given a connected graph $H$ with a vertex of degree $1$, the graph factor $\gamma_H(\mbf{x})$ can be simplified further (since our input graphs are regular).

\begin{lemma}\label{lem:degree1-reduction}
Given a graph $H$ (with no isolated vertices) with a vertex of degree $1$ then $\gamma_{H}(\mbf{x})$ can be expressed, as a function on $d$-regular graphs, as a sum of $\gamma_{H'}(\mbf{x})$ with $v(H')<v(H)$. Furthermore the coefficients of the sum are bounded by $O(1/(p(1-p))^{O_H(1)})$.
\end{lemma}
\begin{proof}
Let $v$ be a vertex in $H$ of degree $1$ and $(u,v)$ be the unique edge in $H$ connected to $v$. Notice that, considering this as a sum over possible choices of $v$, we have $\sum_{v\neq u}\chi_{(v,u)}^{} = 0$ by $d$-regularity. Therefore it follows that
\begin{align*}
\gamma_H(\mbf{x}) &= \sum_{\substack{E\subseteq K_n\\E\simeq H}}\prod_{e\in E}\chi_e^{}= \sum_{\substack{E\subseteq K_n\\E\simeq H}}\chi_{(u,v)}^{}\prod_{e\in E\setminus\{(u,v)\}}\chi_e^{}\\
&= \sum_{\substack{E\subseteq K_n\\E\simeq H}}\bigg(-\sum_{u\in V(E)\setminus\{v\}}\chi_{(u,v)}^{}\bigg)\prod_{e\in E\setminus\{(u,v)\}}\chi_e^{}
\end{align*}
and the desired result follows immediately using that $\chi_e^2 = 1 - (2p-1)\chi_e^{}/\sqrt{p(1-p)}$.
\end{proof}

Note that iterating \cref{lem:connected-reduction,lem:degree1-reduction} shows we can write any graph factor on $d$-regular graphs as a function (in terms of $d$) of ones that are connected and with minimum degree at least $2$. (In particular, any graph factor corresponding to a tree can be expressed in terms of graph factors with cycles as well as the constant $\gamma_\emptyset=1$.) However, we will not explicitly need this fact, but rather its implication that variances of graph factors satisfy a reasonable uniform bound. Furthermore, having a degree $1$ vertex (such as with trees) leads to a natural power-saving in this bound.

\begin{lemma}\label{lem:variance-reduction}
Suppose that $n/\log n\le\min(d,n-d)$. Given a graph $H$ (with no isolated vertices) we have $\mr{Var}_{G\sim\mb{G}(n,d)}(\gamma_H(G)) = O(n^{v(H)})$. Furthermore if $H$ has a degree $1$ vertex, we have $\mr{Var}_{G\sim\mb{G}(n,d)}(\gamma_H(G))\le n^{v(H)-2/3}$.
\end{lemma}
\begin{proof}
We induct on $v(H)$. Note $v(H)\le 2$ is trivial, as in fact $\gamma_H(G)$ is deterministic, so both parts of the lemma are satisfied. For $H$ being a connected graph with minimum degree at least $2$, the desired result follows immediately from the moments calculation given in the proof of \cref{thm:main}. In the remaining cases, for $G$ a $d$-regular graph we find that if $H$ has connected components $H_1,\ldots,H_k$, then $\gamma_H(G)-\prod_{i=1}^k\gamma_{H_i}(G)$ can be written as a sum of graph factors each involving coefficients bounded by $1/(p(1-p))^{O_H(1)}$ and with at most $v(H)-1$ vertices by \cref{lem:connected-reduction}. If there is a vertex of degree $1$ in $H$, and hence some $H_i$, we can apply \cref{lem:degree1-reduction} and then we obtain a sum of graph factors with at most $v(H)-1$ vertices after expanding (with similar bounds on coefficients). Thus the total variance, by induction, is $(p(1-p))^{-O_H(1)}\cdot O(n^{v(H)-1})\le n^{v(H)-2/3}$, which satisfies the desired stronger bound in the case where $H$ has a degree $1$ vertex.

Finally, if all the $H_i$ are minimum degree at least $2$, then we see that the ``lower'' portion corresponding to graph factors on at most $v(H)-1$ vertices contributes $O(n^{v(H)-2/3})$ by induction similar to before. Hence the problem reduces to understanding the variance of $\prod_{i=1}^k\gamma_{H_i}(G)$. We have
\[\mr{Var}\Big(\prod_{i=1}^k\gamma_{H_i}(G)\Big)\le\mb{E}\prod_{i=1}^k\gamma_{H_i}(G)^2\le\prod_{i=1}^k(\mb{E}\gamma_{H_i}(G)^{2k})^{1/k}.\]
Again, the moment-based proof of \cref{thm:main} gives a bound of $O(n^{v(H_1)+\cdots+v(H_k)})$ for this.
\end{proof}

\section{Deduction of Subgraph Count and Trace Count Normality}\label{sec:final}
We now consider a subgraph count $X_H$ for $G\sim\mb{G}(n,d)$ and prove the desired normality as in \cref{cor:deduction-1}. This is essentially an immediate consequence of \cref{thm:main} and expanding into the appropriate graph factors. The precise nature of the contributing terms however depends in an intricate manner on the precise structure of $H$.

\begin{proof}[Proof of \cref{cor:deduction-1}]
Let $H$ be a connected graph at least $2$ vertices which is not a star. For $G\sim\mb{G}(n,d)$ write
\[W = X_H = \sum_{\substack{H'\subseteq K_n\\H'\simeq H}}\prod_{e\in E(H')}x_e.\]
Letting $\chi_e^{} = (x_e-p)/\sqrt{p(1-p)}$ as usual, we find that
\begin{equation}\label{eq:W-subgraph-count}
W=\sum_{\substack{H'\subseteq K_n\\H'\simeq H}}\prod_{e\in E(H')}(p+\sqrt{p(1-p)}\chi_e^{})=\sum_{S\subseteq H}p^{e(H)-e(S)}(\sqrt{p(1-p)})^{e(S)}c_{S,H}d_{S,H}\binom{n-v(S)}{v(H)-v(S)}\gamma_S(\mbf{x}),
\end{equation}
where $c_{S,H}=(v(H)-v(S))!\mr{aut}(S)/\mr{aut}(H)$, $d_{S,H}=N(H,S)$ (the number of times $S$ appears as a subgraph of $H$), and the sum is over subgraphs $S$ (lacking isolated vertices) of $H$ up to isomorphism. For the empty graph, we have $c_{\emptyset,H}=v(H)!/\mr{aut}H$ and $d_{\emptyset,H}=1$.

Notice that the graph factors $\gamma_S$ with $e(S)\le 2$ (the empty graph, an edge, a star with two edges, and two disjoint edges) are deterministic since $G$ is a $d$-regular graph. If $H$ contains a $C_3$ notice that all other graph factors $\gamma_S$ in the expansion have $v(S)\ge 4$ hence the corresponding terms have variance bounded by $O(n^{2(v(H)-v(S))}\cdot n^{v(S)})=O(n^{2v(H)-4})$ by \cref{lem:variance-reduction} while the $\gamma_{C_3}$ term has variance
\[\Bigg(\frac{6(v(H)-3)!}{\mr{aut}(H)}N(H,C_3)p^{e(H)-3/2}(1-p)^{3/2}\binom{n-3}{v(H)-3}\Bigg)^2\mr{Var}[\gamma_{C_3}].\]
Since the variance determination in \cref{thm:main} allows us to compute $\mr{Var}[\gamma_{C_3}]=(1+O(n^{-1/6}))n^3/6$, we easily obtain the first bullet point of \cref{cor:deduction-1}: we can write $W=X+Y$ where $X$ is the term coming from $\gamma_{C_3}$ and $Y$ is the rest. We have that $\mr{Var}[Y]=O(n^{-1}\mr{Var}[X])$. Thus since $X$ satisfies a central limit theorem, so does $X+Y$. Furthermore, the variance can be written
\[\mr{Var}[X+Y]=\mr{Var}[X]+\mr{Var}[Y]+2\mb{E}(X-\mb{E}X)(Y-\mb{E}Y)\]
and $|\mb{E}(X-\mb{E}X)(Y-\mb{E}Y)|\le\sqrt{\mr{Var}[X]\mr{Var}[Y]}$ by Cauchy--Schwarz, which gives an appropriate bound for the change in the variance going from $X$ to $X+Y$. In particular, $\mr{Var}[X+Y]=(1+O(n^{-1/2}))\mr{Var}[X]$.

Next suppose that $H$ contains a $C_4$ but no $C_3$. Then all potential contributing graph factors which are not deterministic are on at least $4$ vertices. Notice that any graph factor $\gamma_S$ with $v(S)\ge 5$ has corresponding variance at most $O(n^{2v(H)-5})$ by \cref{lem:variance-reduction}. Furthermore for $v(S)=4$, notice that if some vertex has degree $1$ then by \cref{lem:variance-reduction} we obtain corresponding variance of order $O(n^{2v(H)-4-2/3})$. All remaining $S$ must have $4$ vertices, minimum degree at least $2$, and contain no $C_3$, so $S=C_4$. The variance of the $\gamma_{C_4}$ term is
\[\Bigg(\frac{8(v(H)-4)!}{\mr{aut}(H)}N(H,C_4)p^{e(H)-2}(1-p)^2\binom{n-4}{v(H)-4}\Bigg)^2\mr{Var}[\gamma_{C_4}].\]
The second bullet point of \cref{cor:deduction-1} follows similar to before.

The last case is when $H$ contains neither $C_3$ nor $C_4$. Since $H$ is not a star, $H$ contains a $P_4$, i.e., a path on $4$ vertices. First note that graph factors $\gamma_S$ with $v(S)\ge 6$ have corresponding variance $O(n^{2v(H)-6})$ and graph factors $\gamma_S$ with $v(S)=5$ and some vertex of degree $1$ have corresponding variance $O(n^{2v(H)-5-2/3})$ by \cref{lem:variance-reduction}. Furthermore, since $H$ has no $C_3$ and no $C_4$, we see that the only possible $S$ with $v(S)\le 4$ for which $\gamma_S$ is not deterministic is $S=P_4$. Also, the possible $S$ with $v(S)=5$ are those with minimum degree at least $2$ and no $C_3$ and no $C_4$, which is easily seen to force $S=C_5$. The variance of the $\gamma_{C_5}$ term is
\begin{align*}
\Bigg(\frac{10(v(H)-5)!}{\mr{aut}(H)}&N(H,C_5)p^{e(H)-5/2}(1-p)^{5/2}\binom{n-5}{v(H)-5}\Bigg)^2\mr{Var}[\gamma_{C_5}]\\
&= (1+O(n^{-1/6}))\frac{10N(H,C_5)^2}{\mr{aut}(H)^2}p^{2e(H)-5}(1-p)^5n^{2v(H)-5}
\end{align*}
by \cref{thm:main}. The graph factor $\gamma_{P_4}$ is more delicate, as we must use the observation in \cref{lem:degree1-reduction} that we can reduce its complexity using $\sum_{v\neq u}\chi_{(v,u)}^{}=0$ for all fixed $u$. We obtain
\begin{align*}
\gamma_{P_4}(\mbf{x}) &= \frac{1}{2}\sum_{u,v,w}\bigg(\chi_{(u,v)}^{}\chi_{(v,w)}^{}\sum_{u'\neq u,v,w}\chi_{(w,u')}^{}\bigg) = \frac{1}{2}\sum_{u,v,w}\chi_{(u,v)}^{}\chi_{(v,w)}^{}(-\chi_{(w,u)}^{}-\chi_{(w,v)}^{})\\
&= -\frac{1}{2}\sum_{u,v,w}(\chi_{(u,v)}^{}\chi_{(v,w)}^{}\chi_{(w,u)}^{}+\chi_{(u,v)}^{}\chi_{(v,w)}^2)\\
&=-3\gamma_{C_3}-\frac{1}{2}\sum_{u,v,w}\chi_{(u,v)}^{}\bigg(1-\frac{(2p-1)\chi_{(v,w)}^{}}{\sqrt{p(1-p)}}\bigg),
\end{align*}
where the sum is over tuples of distinct $u,v,w\in[n]$. Here we have used that $\chi_e^2=1-(2p-1)\chi_e^{}/\sqrt{p(1-p)}$ and given that $\gamma_{K_2},\gamma_{P_3}$ are deterministic, we find that $\gamma_{P_4}+3\gamma_{C_3}$ is deterministic. This can alternatively be deduced by noting that $X_{P_4} + 3X_{C_3}$ is a deterministic function in a $d$-regular graph\footnote{We thank the referee for this remark, which provides a check for the formulas in \cref{cor:deduction-1} since we must have $\mr{Var}[X_{P_4}] = 9\mr{Var}[X_{C_3}].$}. Therefore the variance of the $\gamma_{P_4}$ term is
\begin{align*}
\bigg(\frac{2(v(H)-4)!}{\mr{aut}(H)}&N(H,P_4)p^{e(H)-3/2}(1-p)^{3/2}\binom{n-4}{v(H)-4}\bigg)^2\cdot 9\mr{Var}[\gamma_{C_3}]\\
&= (1+O(n^{-1/6}))\frac{6N(H,P_4)^2}{\mr{aut}(H)^2}p^{2e(H)-3}(1-p)^3n^{2v(H)-5}.
\end{align*}
Finally, writing $X_1$ for the $\gamma_{C_5}$ term and $X_2$ for the $\gamma_{P_4}$ term, using the moment computations in the proof of \cref{thm:main} (applied to $C_3$ and $C_5$) we easily find that $\mr{Cov}(X_1,X_2)=O(n^{2v(H)-5-1/6})$. (Or, we can directly see this from the joint central limit theorem satisfied by $\gamma_{C_3},\gamma_{C_5}$ in \cref{thm:main}.) The third bullet of \cref{cor:deduction-1} follows similar to before.
\end{proof}

Finally, we prove \cref{cor:deduction-2}. Again, this is mostly rearranging terms in order to apply \cref{thm:main}. Our analysis is more complicated than typical trace expansion arguments as one cannot trivially rule out walks where an edge appears with multiplicity $1$ in various expectation computations. To perform the necessary analysis, we will need the following modified version of \cref{lem:graph-theory} which allows for some of the $H_i$ to be a doubled edge (but we may otherwise restrict to cycles). As a consequence, the equality case is more complicated. Recall that for a multigraph $G$, $E_\mr{sing}(G)$ is the set of edges of multiplicity $1$.
\begin{lemma}\label{lem:graph-theory-2}
Let $\mc{H} = (H_i)_{1\le i\le k}$ be a sequence of cycles or multigraphs consisting of a doubled edge. Consider overlaying the $H_i$ in order to obtain a multigraph $G$. Let $E_\mr{sing}=E_\mr{sing}(G)$. Suppose that every connected component of $G$ contains at least one participating cycle of $\mc{H}$. Then we have
\[v(G)-\frac{1}{2}|E_\mr{sing}|\le\frac{1}{2}\sum_{i=1}^kv(H_i)=\frac{e(G)}{2}\]
with equality only if every connected component of $G$ can be obtained by first taking a cycle of $\mc{H}$ or perfectly overlaying two cycles of $\mc{H}$, and second attaching pendant trees of doubled edges from $\mc{H}$ (in particular, removing the cycle portion leaves a forest of doubled edges). Here $e(G)$ is computed with multiplicity.
\end{lemma}
\begin{remark}
We note that for any $H_i$ which is a doubled edge, the corresponding multiedges in $G$ are not contained in $E_\mr{sing}$.
\end{remark}
\begin{proof}
Without loss of generality we may assume $G$ is connected, as this clearly preserves the inequality as well as equality cases. Also, the equality $\sum_{i=1}^kv(H_i)=e(G)$ is trivial since cycles and doubled edges have the same edge and vertex counts. Now let $H_1,\ldots,H_{k'}$ be the cycles and the rest the doubled edges. Let $G'$ be the multigraph overlay of $H_1,\ldots,H_{k'}$ and define $E_\mr{sing}'=E_\mr{sing}(G')$. These are the edges contained in a single $H_i$ for $i\in[k']$. By \cref{lem:graph-theory}, we have
\[v(G')-\frac{1}{2}|E_\mr{sing}'|\le\frac{1}{2}\sum_{i=1}^{k'}v(H_i)\]
and equality can only occur if every connected component of $G'$ is either a single cycle $H_i$ for $i\in[k']$ or an overlay of $2$ equal cycles $H_i,H_{i'}$ for distinct $i,i'\in[k']$. Furthermore, by initial assumption $k'\ge 1$ so $G'$ is nonempty.

Now consider adding in the doubled edges in a specified order, starting at $G_{k'}=G'$ and ending at $G_k=G$. We choose the order as follows: at time $k'\le i\le k-1$, once we have $G_i$, since we know $G$ is connected there must be a doubled edge to add which shares a vertex with $G_i$; add one of those edges. Define $E_\mr{sing}^{(i)}$ in the obvious way. We see that
\[(v(G_{i+1})-v(G_i))-\frac{1}{2}(|E_{\mr{sing}}^{(i+1)}|-|E_\mr{sing}^{(i)}|)\le 1=\frac{1}{2}v(H_{i+1})\]
since either we add $0$ vertices and at worst reduce the number of singleton edges by $1$, or we add $1$ vertex and thus leave the number of singleton edges unchanged (here we are using that $G'$ is nonempty and the connected components of $G_i$ each contain a cycle, otherwise it could be possible to add $2$ vertices at the beginning). Equality occurs here only if we add $1$ new pendant vertex.

Adding these inequalities over all $i$, we obtain the desired inequality. Furthermore, equality can only occur if the connected components of $G'$ are single or doubled cycles, and then we only add pendant trees of doubled edges. But since the final multigraph $G$ is connected, this means we must have started with at most one component as we cannot bridge between connected components of some $G_i$ using a doubled edge while simultaneously increasing the vertex count by $1$. The result follows.
\end{proof}

Finally, we demonstrate \cref{cor:deduction-2}. 

\begin{proof}[Proof of \cref{cor:deduction-2}]
Notice that deterministically we have that the all $1$ vector is an eigenvector with eigenvalue $d$. Therefore we have that $M:=A_G-pJ+pI$ (where $J$ is the all $1$ matrix) has eigenvalues $\lambda_i + p$ for $2\le i\le n$ and one eigenvalue of $0$. 

In order to prove \cref{cor:deduction-2} it suffices to prove that if $E_\ell^\ast = \mb{E}\on{tr}(M^{\ell})$, $\sigma_\ell^{\ast2} = \mr{Var}[\on{tr}(M^\ell)]$ then 
\[\big(\sigma_\ell^{\ast-1/2}\big(\on{tr}(M^\ell)-E_\ell^\ast\big)\big)_{3\le \ell\le k}\xrightarrow[]{d.}\mc{N}(0,\Sigma_k)\]
and $\sigma_\ell^\ast = \Theta((p(1-p)n)^{\ell/2})$ for fixed $\ell\ge 3$. To see that this implies the desired result note that each term of $\sum_{i=2}^n(\lambda_i + p)^\ell-\sum_{i=2}^n\lambda_i^\ell$ can be represented as a degree at most $\ell-1$ polynomial in $\lambda_i+p$ with coefficients bounded by $O_{k}(p^{O_{k}(1)})$. These terms are lower order due to the order of the variance (and using that the first two moments of the eigenvalues are deterministic).

Given $\ell\ge 3$, note that
\[\on{tr}((A_G-pJ+pI)^\ell/(p(1-p))^{\ell/2}) = \sum_{u_1,\ldots,u_\ell\in[n]}\prod_{i=1}^\ell\chi_{(u_i,u_{i+1})}^{}\]
where we define $\chi_{(u,u)}^{}=0$ and take indices modulo $\ell$. The sum is over closed walks of length $\ell$.

Consider the closed walk $u_1,\ldots,u_\ell$. As $\chi_{(u,u)}^{} = 0$, we have that the walk has no self-loops corresponding to $u_{t+1}=u_t$. The edges traced out thus form a multigraph when superimposed. Let $\mc{W}_\ell$ be the collection of possible isomorphism types of multigraphs and for $(u,v)\in G$ and $G\in\mc{W}_\ell$ let $G(u,v)$ be the multiplicity of $(u,v)$ in $G$. We see
\[\on{tr}((A_G-pJ+pI)^\ell/(p(1-p))^{\ell/2}) = \sum_{G\in\mc{W}_\ell}c_G\bigg(\sum_{\substack{V(G')\subseteq V(K_n)\\G'\simeq G}}\prod_{(u,v)\in G}\chi_{(u,v)}^{G(u,v)}\bigg),\]
were $c_G$ is the number of choices of vertices in $G$ and closed walks of length $\ell$ starting at that vertex and traversing each edge $(u,v)\in G$ in either direction exactly $G(u,v)$ times. If $G$ is a simple graph, the term on the inside is just $\gamma_G(\mbf{x})$. We therefore abusively define \[\gamma_G(\mbf{x})=\sum_{\substack{V(G')\subseteq V(K_n)\\G'\simeq G}}\prod_{(u,v)\in G}\chi_{(u,v)}^{G(u,v)}\]
for multigraphs $G$ without isolated vertices. However, we will later use $\chi_e^2=1-(2p-1)\chi_e^{}/\sqrt{p(1-p)}$ and similar relations for higher powers to reduce to a linear combination of graph factors $\gamma_F$.

Furthermore, every multigraph in $\mc{W}_\ell$ can be decomposed (with multiplicity preserved) into a collection of cycles and doubled edges: move along the walk until the first vertex repetition, then remove a portion corresponding to a doubled edge or cycle, and keep doing this. We can further further turn the doubled edges into a multitree by iteratively removing cycles of doubled edges and turning them into two cycles. Given $G\in\mc{W}_\ell$, let $\mc{H}_G$ be the sequence of multigraphs thus generated. Let $\mc{T}_\ell$ be the collection of $G\in\mc{W}_\ell$ that are composed only of doubled edges, which therefore compose a tree as $G$ is connected.

First consider $G\in\mc{W}_\ell\setminus\mc{T}_\ell$, so that $\mc{H}_G$ contains at least one cycle. Let $\mc{W}_G'$ be all possible isomorphism classes $G_1\cup G_2$ for the multigraph union of two copies $G_1,G_2\simeq G$. We see
\[\mr{Var}[\gamma_G]\le\mb{E}\gamma_G^2\lesssim\sum_{G'\in\mc{W}_G'}n^{v(G')}\cdot n^{-|E_\mr{sing}(G')|/2+1/3}\]
by expansion and \cref{prop:graph-factor-computation}. Now consider the collection of cycles and doubled edges which make up $G'$. Since they are overlaid in a way that form two copies of $G$, our condition on $\mc{H}_G$ implies that every connected component of $G'$ has at least one cycle participating in its creation. Thus \cref{lem:graph-theory-2} applies. For cases where equality does not hold, we have $v(G')-|E_\mr{sing}(G')|/2<\ell$ since $e(G')=2\ell$. This implies $v(G')-|E_\mr{sing}(G')|/2+1/3\le\ell-1/6$. For cases where equality does hold, by \cref{lem:graph-theory-2} every connected component of $G'$ must consist of a cycle or doubled cycle (which come from our specified list of cycles that create $G_1,G_2$) and then pendant trees of doubled edges. Note that $G_1,G_2$ are each connected, so we either have that these are disjoint and of this form or they are connected and form such a graph. In the former case $G$ is clearly either a cycle or doubled cycle with pendant trees of doubled edges. In the latter case we easily deduce that $G$ is a single cycle with pendant trees of doubled edges (recalling $G\notin\mc{T}_\ell$). Let $\mc{C}_\ell$ be the set of isomorphism classes of these more special forms, so that $\mr{Var}[\gamma_G]=O(n^{\ell-1/6})$ for $G\in\mc{W}_\ell\setminus(\mc{T}_\ell\cup\mc{C}_\ell)$.

Next we study $G\in\mc{T}_\ell$, in which case $\ell$ must be even (thus $\ell\ge 4$) and $G$ has $\ell/2$ doubled edges and at most $\ell/2+1$ vertices (being a multitree). We write
\[\gamma_G(\mbf{x})=\sum_{F\subseteq G}c_{F,G}(p)\bigg(\sum_{\substack{F'\subseteq K_n\\F'\simeq F}}\prod_{(u,v)\in F'}\chi_{(u,v)}^{}\bigg)\]
where the sum is over graphs $F$ up to isomorphism obtained by either including $1$ or $0$ edges for each edge in $G$ (without multiplicity). Here $c_{F,G}(p)$ are appropriately computed constants. This is shown by expanding via $\chi_e^2=1-(2p-1)\chi_e^{}/\sqrt{p(1-p)}$ and similar for higher powers, and collecting the patterns that can result. Note also that $F$ may have isolated vertices, and that it is a forest since $G$ is a multitree. Regardless of these isolated vertices, let us abusively denote the inside term as $\gamma_F(\mbf{x})$ (this agrees with the usual definition). If $e(F)\le 2$ then the term corresponding to $F$ is deterministic since we are considering $d$-regular graphs. Hence we may restrict to just terms with $e(F)\ge 3$. We have for such $F$ that
\[\mr{Var}[\gamma_F]\le\mb{E}\gamma_F^2\lesssim\max_{v\ge 0}n^{2(\ell/2+1)-v}\cdot n^{-(2e(F)-\max(v-1,0))/2+1/3} = \max_{v\ge 0}n^{\ell+2-e(F)+\max(-v-1,-2v)/2+1/3}\]
by \cref{prop:graph-factor-computation}: two copies of $F$ with $v$ overlapping vertices can share at most $\max(v-1,0)$ edges. This clearly yields $\mr{Var}[\gamma_F]=O(n^{\ell-2/3})$, and thus we find $\mr{Var}[\gamma_G]=O(n^{\ell-1/6})$ for $G\in\mc{T}_\ell$.

Finally, consider $G\in\mc{C}_\ell$. In $\gamma_G(\mbf{x})$, the highest degree of any term $\chi_e^{}$ is $2$. Using $\chi_e^2=1-(2p-1)\chi_e^{}/\sqrt{p(1-p)}$ and expanding out, it is easy to see similar to above that the sum of the terms involving any $-(2p-1)\chi_e^{}/\sqrt{p(1-p)}$ in the expansion, call this $\gamma_G'$, has total variance bounded by $O(n^{\ell-1/6})$. Finally, if $G$ is a doubled cycle with pendant trees of double edges then the remaining term is deterministic, while if $G$ is a single cycle with such pendant trees then the remaining term is a cycle of say length $r$ with $(\ell-r)/2$ isolated vertices. Finally, recall that the variance of $\gamma_{C_r}$ is $O(n^r)$ by \cref{thm:main}.

Overall, combining all this information and noting that a $\gamma_{C_\ell}$ term only comes from a walk that repeats no vertices, we see
\[\on{tr}((A_G-pJ+pI)^\ell/(p(1-p))^{\ell/2}) = 2\ell\gamma_{C_\ell}(\mbf{x})+\sum_{\substack{3\le r<\ell\\r\equiv\ell\pmod{2}}}\alpha_{\ell,r}n^{(\ell-r)/2}\gamma_{C_r}(\mbf{x}) + X_\ell\]
for some random variable $X_\ell$ satisfying $\mr{Var}[X_\ell]=O(n^{\ell-1/6})$ and for appropriate combinatorially definable rational numbers $\alpha_{\ell,r}$ independent of $n,p$.

Equivalently,
\[\on{tr}((A_G-pJ+pI)^\ell/(p(1-p)n)^{\ell/2})=2\ell(n^{-\ell/2}\gamma_{C_\ell})+\sum_{\substack{3\le r<\ell\\r\equiv\ell\pmod{2}}}\alpha_{\ell,r,n}(n^{-r/2}\gamma_{C_r}) + n^{-\ell/2}X_\ell\]
and now the result clearly follows from \cref{thm:main} as the error terms $X_\ell$ are negligible (using a similar argument as in the proof of \cref{cor:deduction-1}). Since this representation in terms of the cycle graph factors is triangular, we furthermore see that the resulting $\Sigma_k$ that arises in the limit is indeed positive definite; here we are using that the coefficient of $\gamma_{C_{\ell}}$ is a strictly positive constant, that the $\alpha_{\ell,r}$ are independent of $p,n$ and of bounded size in terms of $\ell$, and that the graph factors corresponding to cycles are jointly independently normally distributed by \cref{thm:main}.
\end{proof}

\bibliographystyle{amsplain0-full.bst}
\bibliography{main.bib}

\providecommand{\bysame}{\leavevmode\hbox to3em{\hrulefill}\thinspace}
\providecommand{\MR}{\relax\ifhmode\unskip\space\fi MR }
\providecommand{\MRhref}[2]{%
  \href{http://www.ams.org/mathscinet-getitem?mr=#1}{#2}
}
\providecommand{\href}[2]{#2}
\begin{thebibliography}{10}

\bibitem{Bol80}
B\'{e}la Bollob\'{a}s, \emph{A probabilistic proof of an asymptotic formula for
  the number of labelled regular graphs}, European Journal of Combinatorics
  \textbf{1} (1980), 311--316.

\bibitem{Dur19}
Rick Durrett, \emph{Probability---theory and examples}, Cambridge Series in
  Statistical and Probabilistic Mathematics, vol.~49, Cambridge University
  Press, Cambridge, 2019, Fifth edition.

\bibitem{Gao20}
Pu~Gao, \emph{Triangles and subgraph probabilities in random regular graphs}.

\bibitem{GW08}
Zhicheng Gao and N.~C. Wormald, \emph{Distribution of subgraphs of random
  regular graphs}, Random Structures \& Algorithms \textbf{32} (2008), 38--48.

\bibitem{He22}
Yukun He, \emph{Spectral gap and edge universality of dense random regular
  graphs}.

\bibitem{IM18}
Mikhail Isaev and Brendan~D. McKay, \emph{Complex martingales and asymptotic
  enumeration}, Random Structures \& Algorithms \textbf{52} (2018), 617--661.

\bibitem{Jan94b}
Svante Janson, \emph{The numbers of spanning trees, {H}amilton cycles and
  perfect matchings in a random graph}, Combinatorics, Probability and
  Computing \textbf{3} (1994), 97--126.

\bibitem{Jan94}
Svante Janson, \emph{Orthogonal decompositions and functional limit theorems
  for random graph statistics}, Memoirs of the American Mathematical Society
  \textbf{111} (1994), vi+78.

\bibitem{Jan95}
Svante Janson, \emph{A graph {F}ourier transform and proportional graphs},
  Proceedings of the {S}ixth {I}nternational {S}eminar on {R}andom {G}raphs and
  {P}robabilistic {M}ethods in {C}ombinatorics and {C}omputer {S}cience,
  ``{R}andom {G}raphs '93'' ({P}ozna\'{n}, 1993), vol.~6, 1995, pp.~341--351.

\bibitem{JLR00}
Svante Janson, Tomasz {\L}uczak, and Andrzej Rucinski, \emph{Random graphs},
  Wiley-Interscience Series in Discrete Mathematics and Optimization,
  Wiley-Interscience, New York, 2000.

\bibitem{KSV07}
Jeong~Han Kim, Benny Sudakov, and Van Vu, \emph{Small subgraphs of random
  regular graphs}, Discrete Mathematics \textbf{307} (2007), 1961--1967.

\bibitem{LS20}
Benjamin Landon and Philippe Sosoe, \emph{Applications of mesoscopic {CLT}s in
  random matrix theory}, The Annals of Applied Probability \textbf{30} (2020),
  2769--2795.

\bibitem{LW17}
Anita Liebenau and Nick Wormald, \emph{Asymptotic enumeration of graphs by
  degree sequence, and the degree sequence of a random graph}.

\bibitem{McK85}
Brendan~D. McKay, \emph{Asymptotics for symmetric {$0$}-{$1$} matrices with
  prescribed row sums}, Ars Combinatoria \textbf{19} (1985), 15--25.

\bibitem{McK10}
Brendan~D. McKay, \emph{Subgraphs of random graphs with specified degrees},
  Proceedings of the {I}nternational {C}ongress of {M}athematicians. {V}olume
  {IV}, Hindustan Book Agency, New Delhi, 2010, pp.~2489--2501.

\bibitem{McK11}
Brendan~D. McKay, \emph{Subgraphs of dense random graphs with specified
  degrees}, Combinatorics, Probability and Computing \textbf{20} (2011),
  413--433.

\bibitem{MW90}
Brendan~D. McKay and Nicholas~C. Wormald, \emph{Asymptotic enumeration by
  degree sequence of graphs of high degree}, European Journal of Combinatorics
  \textbf{11} (1990), 565--580.

\bibitem{MW91}
Brendan~D. McKay and Nicholas~C. Wormald, \emph{Asymptotic enumeration by
  degree sequence of graphs with degrees {$o(n^{1/2})$}}, Combinatorica. An
  International Journal on Combinatorics and the Theory of Computing
  \textbf{11} (1991), 369--382.

\bibitem{MWW04}
Brendan~D. McKay, Nicholas~C. Wormald, and Beata Wysocka, \emph{Short cycles in
  random regular graphs}, Electronic Journal of Combinatorics \textbf{11}
  (2004), Research Paper 66, 12.

\bibitem{RW19}
Philippe Rigollet and Jonathan Weed, \emph{Uncoupled isotonic regression via
  minimum {W}asserstein deconvolution}, Information and Inference. A Journal of
  the IMA \textbf{8} (2019), 691--717.

\bibitem{Ru88}
Andrzej Ruci\'{n}ski, \emph{When are small subgraphs of a random graph normally
  distributed?}, Probability Theory and Related Fields \textbf{78} (1988),
  1--10.

\bibitem{SS98}
Ya. Sinai and A.~Soshnikov, \emph{Central limit theorem for traces of large
  random symmetric matrices with independent matrix elements}, Boletim da
  Sociedade Brasileira de Matem\'{a}tica. Nova S\'{e}rie \textbf{29} (1998),
  1--24.

\bibitem{Wor18}
Nicholas Wormald, \emph{Asymptotic enumeration of graphs with given degree
  sequence}, Proceedings of the {I}nternational {C}ongress of
  {M}athematicians---{R}io de {J}aneiro 2018. {V}ol. {IV}. {I}nvited lectures,
  World Sci. Publ., Hackensack, NJ, 2018, pp.~3245--3264.

\bibitem{Wor81}
Nicholas~C. Wormald, \emph{The asymptotic distribution of short cycles in
  random regular graphs}, Journal of Combinatorial Theory. Series B \textbf{31}
  (1981), 168--182.

\end{thebibliography}


\end{document}